\documentclass[12pt]{amsart}

\usepackage[left=2.5cm,right=2.5cm,top=2.5cm, bottom=2.5cm]{geometry}

\usepackage{float}

\usepackage{amscd}
\usepackage{amssymb}
\usepackage{mathrsfs}
\usepackage{amsmath}
\usepackage{mathabx}
\usepackage{amsthm}
\usepackage[all,cmtip]{xy}
\usepackage{enumerate}
\usepackage{enumitem}
\usepackage{ytableau}
\usepackage{tikz}
\usepackage{tikz-cd} 
\usepackage{mathtools}
\usepackage{color}
\usepackage{wasysym}

\RequirePackage{ae, aecompl, aeguill} 

\usepackage[plainpages,urlcolor=red,linktocpage=true]{hyperref}

\numberwithin{equation}{section}

\newcommand{\thmref}[1]{Theorem~\ref{#1}}
\newcommand{\secref}[1]{Section~\ref{#1}}

\newcommand{\lemref}[1]{Lemma~\ref{#1}}
\newcommand{\propref}[1]{Proposition~\ref{#1}}
\newcommand{\corref}[1]{Corollary~\ref{#1}}
\newcommand{\defref}[1]{Definition~\ref{#1}}
\newcommand{\remref}[1]{Remark~\ref{#1}}
\newcommand{\eqnref}[1]{(\ref{#1})}
\newcommand{\exref}[1]{Example~\ref{#1}}
\newcommand{\figref}[1]{Figure~\ref{#1}}
\newcommand{\rmkref}[1]{Remark~\ref{#1}}

\newtheorem{theorem}{Theorem}[section]
\newtheorem{lemma}[theorem]{Lemma}
\newtheorem{proposition}[theorem]{Proposition}
\newtheorem{corollary}[theorem]{Corollary}

\theoremstyle{definition}
\newtheorem{example}[theorem]{Example}

\newtheorem{definition}[theorem]{Definition}

\theoremstyle{remark}
\newtheorem{remark}[theorem]{Remark}

\theoremstyle{plain} 
\newcommand{\thistheoremname}{}
\newtheorem{genericthm}[theorem]{\thistheoremname}

\newtheorem*{genericthm*}{\thistheoremname}
\newenvironment{namedthm*}[1]
  {\renewcommand{\thistheoremname}{#1}%
   \begin{genericthm*}}
  {\end{genericthm*}}

\newcommand{\nc}{\newcommand}

\nc{\on}{\operatorname}

\nc{\Z}{{\mathbb Z}}
\nc{\bQ}{{\mathbb Q}}
\nc{\C}{{\mathbb C}}
\nc{\R}{{\mathbb R}}
\nc{\bbP}{{\mathbb P}}
\nc{\bF}{{\mathbb F}}

\nc{\boldD}{{\mathbb D}}
\nc{\oo}{{\mf O}}
\nc{\N}{{\mathbb N}}
\nc{\bib}{\bibitem}
\nc{\pa}{\partial}
\nc{\F}{{\mf F}}
\nc{\CA}{{\mathcal A}}
\nc{\cD}{{\mathcal D}}
\nc{\CE}{{\mathcal E}}
\nc{\CP}{{\mathcal P}}
\nc{\CO}{{\mathcal O}}
\nc{\CU}{{\mathcal U}}
\nc{\Res}{\text{Res}}
\nc{\Ind}{\text{Ind}}
\nc{\Ker}{\text{Ker}}
\nc{\id}{\text{id}}
\nc{\ot}{\otimes}

\nc{\be}{\begin{equation}}
\nc{\ee}{\end{equation}}

\nc{\rarr}{\rightarrow}
\nc{\larr}{\longrightarrow}
\nc{\al}{\alpha}
\nc{\ri}{\rangle}
\nc{\lef}{\langle}

\nc{\W}{{\mc W}}
\nc{\gam}{\ol{\gamma}}
\nc{\Q}{\ol{Q}}
\nc{\q}{\widetilde{Q}}
\nc{\la}{\lambda}
\nc{\ep}{\epsilon}
\nc{\ve}{\varepsilon}

\nc{\g}{\mf g}
\nc{\h}{\mf h}
\nc{\n}{\mf n}
\nc{\bb}{\mf b}
\nc{\G}{{\mf g}}

\nc{\D}{\mc D}
\nc{\cE}{\mc E}
\nc{\CC}{\mc C}
\nc{\CH}{\mc H}
\nc{\CK}{\mc K}
\nc{\CT}{\mc T}
\nc{\CI}{\mc I}
\nc{\CR}{\mc R}

\nc{\UK}{{\mc U}_{\CA_q}}

\nc{\CS}{\mc S}

\nc{\CB}{\mc B}

\nc{\Li}{{\mc L}}
\nc{\La}{\Lambda}
\nc{\is}{{\mathbf i}}
\nc{\V}{\mf V}
\nc{\bi}{\bibitem}
\nc{\NS}{\mf N}
\nc{\dt}{\mathord{\hbox{${\frac{d}{d t}}$}}}
\nc{\E}{\mc E}
\nc{\ba}{\tilde{\pa}}
\nc{\half}{\frac{1}{2}}

\nc{\mc}{\mathcal}
\nc{\ov}{\overline}
\nc{\mf}{\mathfrak}
\nc{\ol}{\fracline}
\nc{\el}{\ell}
\nc{\etabf}{{\bf \eta}}
\nc{\zetabf}{{\bf
\zeta}}\nc{\x}{{\bf x}}
\nc{\xibf}{{\bf \xi}} \nc{\y}{{\bf y}}
\nc{\WW}{\mc W}
\nc{\SW}{\mc S \mc W}
\nc{\sd}{\mc S \mc D}
\nc{\hsd}{\widehat{\mc S\mc D}}
\nc{\parth}{\partial_{\theta}}
\nc{\cwo}{\C[w]^{(1)}}
\nc{\cwe}{\C[w]^{(0)}} \nc{\wt}{\widetilde}
\nc{\gl}{\mf gl}
\nc{\K}{\mf k}

\newcommand{\Sym}{{\rm{Sym}}}

\advance\headheight by 2pt

\nc{\fb}{{\mathfrak b}} \nc{\fg}{{\mathfrak g}}
\nc{\tA}{\widetilde{A}}
\nc{\fh}{{\mathfrak h}}  \nc{\fk}{{\mathfrak k}}

\nc{\fl}{{\mathfrak l}} \nc{\fn}{{\mathfrak n}}

\nc{\fp}{{\mathfrak p}} \nc{\fu}{u}

\nc{\fS}{{\Sym}}

\nc{\fsl}{{\mathfrak {sl}}} \nc{\fsp}{{\mathfrak {sp}}}
\nc{\fso}{{\mathfrak {so}}} \nc{\fgl}{{\mathfrak {gl}}}

\nc{\A}{\mc A} \nc{\cF}{{\mathcal F}}

\nc{\cA}{{\mathcal A}} \nc{\cP}{{\mathcal P}} \nc{\cC}{{\mathcal C}}
\nc{\cU}{{\mathcal U}} \nc{\cB}{{\mathcal B}}

\def\xl{{\hbox{\lower 2pt\hbox{$\scriptstyle \mathfrak L$}}}}

\nc{\bX}{{\mathbf X}} \nc{\bx}{{\mathbf x}} \nc{\bd}{{\mathbf d}}
\nc{\bdim}{{\mathbf dim}} \nc{\bm}{{\mathbf m}}

\begin{document}

\title[Homology of the NCP lattice and the Milnor fibre]{On the homology of the noncrossing partition lattice and the Milnor fibre}
\author{Yang Zhang}
\email{yang.zhang@uq.edu.au}

\address{School of Mathematics and Statistics, The University of Sydney, NSW 2006, Australia}
\address{School of Mathematics and Physics, The University of Queensland, QLD 4072,  Australia}

\date{9 June, 2022}
\keywords{noncrossing partition lattice, Milnor fibre, hyperplane arrangement, Artin group}
\subjclass[2010]{20F55,52C35,05E14}

\begin{abstract}
    Let $\mathcal{L}$ be the noncrossing partition lattice associated to a finite Coxeter group $W$.
    In this paper we construct  explicit bases for the top homology groups of  intervals and rank-selected subposets  of $\mathcal{L}$. We define a multiplicative structure on the Whitney homology of $\mathcal{L}$ in terms of the basis, and the resulting algebra has similarities to the Orlik-Solomon algebra.   As an application, we obtain four  chain complexes which compute the integral homology  of the Milnor fibre of the reflection arrangement of $W$,  the Milnor fibre of the discriminant of $W$,  the hyperplane complement of $W$ and  the Artin group of type $W$, respectively. We also tabulate some computational results on the integral homology  of the Milnor fibres.
\end{abstract}
    \maketitle


\section{Introduction}

Let $W$ be a finite Coxeter group acting faithfully on  $\mathbb{R}^n$ as a reflection group, and let $T$ be the set of reflections of $W$. Denote by $\ell_T(w)$ the minimal number of reflections required to express $w$, and define the partial order $\leq$ on $W$  by letting $u\leq v$ if and only if $\ell_T(v)= \ell_T(u)+ \ell_{T}(u^{-1}v)$ for $u,v\in W$. Then $(W,\leq)$ is a graded poset whose unique minimal element is the identity $e$ and the maximal elements are those having no fixed points in $\mathbb{R}^n$ (e.g. Coxeter elements and $-e$ when it is in $W$).   For any Coxeter element $\gamma$,  the closed interval $\mathcal{L}:=[e,\gamma]$ of the poset $(W,\leq)$ is  a lattice \cite{Bes03,BW08}, called noncrossing partition (NCP) lattice. The isomorphism type of $\mathcal{L}$ is independent of the choice of $\gamma$. If $W$ is of type $A_n$,  elements of $\mathcal{L}$ are those partitions of $\{1,2,\dots, n+1\}$ whose blocks do not pairwise cross \cite{Kre72}. 

It has been known that $\mathcal{L}$ has an EL-labelling, and hence is Cohen-Macaulay \cite{ABW07,Bjo}. This implies that the reduced homology groups of any interval or rank-selected subposet of $\mathcal{L}$ vanish except for the top one. In this paper, we describe explicitly the top homology groups of intervals  and rank-selected subposets of $\mathcal{L}$.

Our motivation to study the homology of $\mathcal{L}$ arises from its connection with the Milnor fibre. Let $\mathcal{A}$ be the set of complexified reflecting hyperplanes of $W$ in $V:=\mathbb{C}^n\cong \mathbb{C}\otimes \mathbb{R}^n$, and for each $H\in \mathcal{A}$ let $\lambda_{H}\in V^*$ be a linear form such that $H=\{v\in V\mid  \lambda_H(v)=0\}$. We call $M:=V- \bigcup_{H\in \mathcal{A}}H$ the hyperplane complement of $W$ \cite{OS80,Sal87}. The homogeneous polynomial $Q_0:=\prod_{H\in \mathcal{A}} \lambda_H$ restricts to a locally trivial fibration $Q_0: M \rightarrow \mathbb{C}-\{0\}$, with the typical Milnor fibre $F_{Q_0}:=Q_0^{-1}(1)$ \cite{Mil68}. The polynomial $Q_0$ satisfies $Q_0(wx)={\rm det}(w) Q_0(x)$ for any $w\in W$ and $x\in V$, and hence $Q:=Q_0^2$ is $W$-invariant. This gives rise to a free $W$-action on the Milnor fibre $F_{Q}:=Q^{-1}(1)$. The quotient space $F_{Q}/W$ can be realised as the Milnor fibre $F_{P}:=P^{-1}(1)$ of the discriminant  $P$ of $W$ \cite{CS04,DL16}. 

Brady, Falk and Watt gave a finite simplicial complex which has the homotopy type of the Milnor fibre $F_{Q_0}$ \cite{BFW18}. This simplicial complex is described in terms of the NCP lattice and hence is called the NCP model of  $F_{Q_0}$. Similarly, they introduced the NCP model of $F_P$ which is a CW complex. Using  a filtration by subcomplexes of the NCP model,  they obtained a chain complex which computes the integral homology of $F_P$, with the chain groups being the top homology groups of some rank-selected subposets of $\mathcal{L}$. However, an explicit description of those  top homology groups and boundary maps was in absence \cite[Theorem 6.4]{BFW18}, and the NCP model of $F_{Q_0}$ is  cumbersome for computational purposes. 

The aim of this study is to give a more practical and efficient way to calculate the homology of the Milnor fibres based on the approach through NCP models. This requires a complete description of the homology of subposets of the NCP lattice $\mathcal{L}$. However,  as $\mathcal{L}$ is not a geometric lattice, it is very  difficult to adapt usual techniques (see, e.g. \cite{Bjo82}) to construct the top homology cycles explicitly. An attempt for $W$ of type $A_n$ can be found  in \cite{Zoq06}. 

A new idea here is to make use of the Hurwitz action. For each interval $[e,w]$ of $\mathcal{L}$, the  Hurwitz action is known to yield a transitive action of the braid group $B_{\ell_{T}(w)}$ on the maximal chains of $[e,w]$ \cite{Bes03}.
Given a maximal chain $e=w_0<w_1<\dots <w_k=w$ of $[e,w]$, we identify it with the sequence ${\bf t}:=(t_1, t_2,\dots ,t_k)$ of reflections, where $k=\ell_{T}(w)$ and  $t_i=w_{i-1}^{-1}w_{i}\in T$ for $1\leq i\leq k$.  Associated to this chain we define the element
\[\beta_{{\bf t}}:=\sum_{\pi\in {\rm Sym}_k} {\rm sgn}(\pi)\, \underline{\pi}.(t_1,t_2,\dots,t_k),\]
where ${\rm sgn}$ is the usual sign character  of ${\rm Sym}_k$, $\underline{\pi}$ is the set-theoretic lift of ${\rm Sym}_k$ into the braid group $B_k$, and the action on the sequence is given by the  Hurwitz action (see \eqref{eq: Braidact}). These elements associated to the maximal chains of $[e,w]$ span a free abelian group $\mathcal{B}_w$, which is shown to be isomorphic to the top reduced homology group  of the interval $(e,w)$. One can further define the free abelian group $\mathcal{B}_k:= \bigoplus_{w\in \mathcal{L}_k} \mathcal{B}_w$ for $0\leq k\leq n$, where $\mathcal{L}_k$ comprises elements $w\in \mathcal{L}$ with $\ell_{T}(w)=k$. We give a surjective linear map from $\mathcal{B}_k$ to the top reduced  homology group of the rank-selected subposet $\mathcal{L}_{[k]}$ of $\mathcal{L}$, where $\mathcal{L}_{[k]}$ consists of elements $w\in \mathcal{L}$ such that $1\leq \ell_{T}(w)\leq k$. 
Therefore, the homology of intervals and rank-selected subposets of $\mathcal{L}$ can be described by the $\mathbb{Z}$-graded free abelian group $\mathcal{B}:= \bigoplus_{k=0}^n \mathcal{B}_k$.

Using the construction above, we give explicit bases for the top homology groups of intervals $(e,w)$ and  rank-selected subposets $\mathcal{L}_{[k]}$ in \thmref{thm: homint} and \thmref{thm: basisrksel}, respectively. Moreover, we introduce a natural  multiplicative structure on the free abelian group $\mathcal{B}$, which by our construction is isomorphic to the Whitney homology of $\mathcal{L}$ \cite{Bac75,Bjo82}. This makes $\mathcal{B}$ into a finite dimensional algebra, which is studied systematically in a sequel to this paper \cite{LZ22}. Interestingly,  the algebra $\mathcal{B}$ resembles the Orlik-Solomon algebra of the intersection lattice \cite{OS80},  and its main properties  are summarised in \thmref{thm: main}.

We apply the above results to the NCP models of the Milnor fibres, obtaining two chain complexes which compute the homology groups $H_*(F_Q; \mathbb{Z})$ and $H_*(F_P; \mathbb{Z})$ in terms of the algebra $\mathcal{B}$, respectively; see \thmref{thm: fullcom} and \thmref{thm:  disfibre}.  The chain complex for $F_{Q}$ comes with  an action of the Coxeter group $W$ and a partial action of the monodromy group of $F_{Q}$. A chain complex for $F_{Q_0}$ can  be derived from that for $F_{Q}$. We calculate $H_*(F_{Q_0}; \mathbb{Z})$ and $H_*(F_P; \mathbb{Z})$ by computer for some exceptional types and infinite families $A_n, B_n, D_n$ in low ranks. Most of our computational results on $H_*(F_{Q_0}; \mathbb{Z})$ are new, and in particular we find no torsion  in the obtained homology groups. Moreover, we correct Settepanella's result on $H_*(F_{Q_0}; \mathbb{Q})$ for type $F_4$ \cite{Set09},  and as far as we know, our results on $H_*(F_P;\mathbb{Z})$ for type $D_n, n\leq 8$ are  new.

A complete description of the homology of the Milnor fibre $F_Q$ or $F_{Q_0}$ is not known despite a substantial literature (see, e.g. \cite{CS95,DL12,DL16}). Indeed, there were very few known  results about $H_*(F_{Q_0}; \mathbb{Z})$. The normal approach to  calculating the homology $H_*(F_{Q_0}; R)$ and $H_*(F_P; R)$ is through the Salvetti complex with coefficients in $R[q,q^{-1}]$ for a ring $R$ \cite{DPS01,CS04,Cal06,Set04,Set09}. It is usually difficult to obtain the abelian group structure of the homology  when $R=\mathbb{Z}$, the ring of integers. Our approach  provides a direct way to compute the integral homology. 

As another application, we obtain two chain complexes which compute $H_*(M; \mathbb{Z})$ and $H_*(M/W; \mathbb{Z})$ in terms of the algebra $\mathcal{B}$ in \thmref{thm: homoM} and \thmref{thm: homoMquo}, respectively.  This is again based on the  NCP models of the hyperplane complement $M$ and its quotient $M/W$ \cite{BFW18}. It is well known that the (co)homology of $M$ is torsion free \cite{Bri73} and  the cohomology ring $H^*(M;\mathbb{Z})$ is isomorphic to the Orlik-Solomon algebra \cite{OS80}. This further suggests a mysterious link between $\mathcal{B}$ and the Orlik-Solomon algebra. Moreover, since $M/W$ is a $K(\pi,1)$ space having the Artin group $A(W)$ as its fundamental group, we have $H_*(M/W; \mathbb{Z})=H_*(A(W);\mathbb{Z})$. Therefore, in addition to the approach of the Salvetti complex \cite{Sal87,Sal94},  our chain complex for $M/W$ provides a new algebraic method to compute the integral homology of $A(W)$ for any finite Coxeter group $W$.

This paper is organised as follows. In \secref{sec: pre} we introduce some basic notions and results. In \secref{sec: cons} we give an explicit construction of the top homology cycles of any interval by using the Hurwitz action. Bases for the top homology groups of  intervals and rank-selected subposets are given in \secref{sec: basis}. In \secref{sec: mult} we define a natural multiplicative structure on  the Whitney homology of $\mathcal{L}$. Application to the Milnor fibres $F_{Q}$ and $F_P$ is discussed in \secref{sec: app1}, and to the hyperplane complement $M$ and its quotient $M/W$  given in \secref{sec: app2}. Finally, in Appendix \ref{sec: appd} we tabulate computational results on the integral  homology of the Milnor fibres $F_{Q_0}$ and $F_P$.

\vspace{0.3cm}
\noindent
{\bf Acknowledgements.} This paper is based on my PhD thesis. I would like to express my deepest gratitude to my supervisors  Gus Lehrer and  Ruibin Zhang for their support, encouragement and many insightful conversations in the course of this work. I am grateful to Anthony Henderson for useful discussions. I would also like to thank the referee, whose comments and suggestions have greatly improved the exposition in this paper. This work is supported by the Australian Research Council.

\section{Preliminaries}\label{sec: pre}

We work over the ring $\mathbb{Z}$ of integers throughout the paper. For any $m\in \mathbb{Z}$, denote by $\mathbb{Z}_m$ the abelian group $\mathbb{Z}/m \mathbb{Z}$. Let $W$ be a finite Coxeter group of rank $n$, and let $T$ be the set of reflections of $W$. For any $u,w\in W$ we write $u^{w}$ for the conjugate $w^{-1}uw$. For any non-negative integer $n$, we denote $[n]:=\{1,2, \dots, n\}$ and $[0]:=\emptyset$. By abuse of notation, we  always denote by $\partial_k$  the  boundary maps of  chain complexes.

\subsection{Basics on posets}
We refer to \cite{Sta12} for basics on posets. Let $(P, \leq)$ be a finite  poset.  Assume that $P$ is bounded, i.e. $P$  contains a unique minimal element $\hat{0}$ and a unique  maximal element $\hat{1}$. The proper part of  $P$  is defined to be $\bar{P}:=P-\{\hat{0}, \hat{1}\}$. For a non-bounded poset $Q$, we may define its bounded extension  $\hat{Q}:=Q\cup \{\hat{0}, \hat{1}\} $.

If every maximal chain of  $P$ has the same length $n$, we call $P$ a \emph{graded poset} of rank $n$. There is a  unique associated rank function 
$ {\rm rk}: P\rightarrow \{0,1,\dots, n\}$
such that ${\rm rk}(\hat{0})=0$ and ${\rm rk}(y)={\rm rk}(x)+1$ if $x\lessdot y$, where $x\lessdot y$ means there is no third element between $x$ and $y$. 

For $x\leq y$ in $P$, we denote by $(x,y)$ the open interval $\{z \in P \mid  x<z<y\}$ and by $[x,y]$ the closed interval $\{ z \in P \mid x\leq z\leq y\}$.  For any subset $S$ of $[n-1]$, we define the rank-selected subposet $P_{S}:=\{x\in P \mid {\rm rk}(x)\in S \}$.

The order complex $\Delta(\bar{P})$  is a simplicial complex which has as faces the chains of $\bar{P}$.  The $k$-th reduced homology group of $\bar{P}$ is defined  to be $\widetilde{H}_k(\bar{P}):= \widetilde{H}_k(\Delta(\bar{P}); \mathbb{Z})$, the reduced simplicial  homology of the order complex $\Delta(\bar{P})$. For any open interval $(x,y)$ of $P$,  we denote by  $\Delta(x,y)$  the order complex of $(x,y)$ and by $\widetilde{H}_k(x,y)$ the $k$-th reduced homology of $\Delta(x,y)$. In particular, we have $\widetilde{H}_{-1}(\emptyset)=\mathbb{Z}$ and $\widetilde{H}_{k}(\emptyset)=0$ for $k>-1$.

Recall that the \emph{ M\"{o}bius function} $\mu=\mu_P: P\times P \rightarrow \mathbb{Z}$ is the  integer-valued function defined by
$\mu(x,x)=1$ for any $x\in P$, $\sum_{x\leq z\leq y }\mu(x,z)=0$ if  $x<y$, and $\mu(x,y)=0$ otherwise.
Denote $\mu(P):=\mu(\hat{0}, \hat{1})$ and $\mu(x):=\mu(\hat{0},x)$ for $x\in P$.

Let $(W,S)$ be a finite Coxeter system of rank $n$, and let $T=\bigcup_{w\in W}wSw^{-1}$ be the set of reflections of $W$. 
Denote by $\ell_T(w)$ the number of reflections in a shortest expression for $w\in W$ as a product of reflections.
 We call such a shortest factorisation  a \emph{$T$-reduced expression} of $w$. 
Note that the absolute length is invariant under conjugation, i.e. $\ell_{T}(uwu^{-1})=\ell_{T}(w)$ for any $u,w\in W$ \cite{Bes03}. 

We define the \emph{absolute order} $\leq$ on $W$  by specifying that $u\leq v$ if and only if $\ell_T(v)= \ell_T(u)+ \ell_{T}(u^{-1}v)$ for $u,v\in W$ \cite{Bes03,BW08}. The Coxeter group  $W$ equipped with the absolute order $\leq$ is  a graded poset with $\ell_{T}$ being the rank function. In particular, the poset $(W,\leq)$ has all  elliptical elements as its maximal elements and the identity $e$ as the unique minimal element \cite[Lemma 1.2.1]{Bes03}. Recall from \cite{Leh05} that the elliptical elements of $W$ are those with no fixed points in $\mathbb{R}^n$, e.g. Coxeter elements and $-e$ when it is in $W$. 

Let $\gamma$ be a Coxeter element, i.e. a product of  simple reflections of $S$ in some order. Following \cite{Bes03,BW08}, we define the \emph{noncrossing partition lattice} to be the closed interval $\mathcal{L}:=[e,\gamma]$ of $(W,\leq)$. Since the Coxeter elements form a conjugacy class \cite{Ste59}, different choices of $\gamma$  produce isomorphic noncrossing partition lattices $\mathcal{L}=[e,\gamma]$. Therefore,  the isomorphism type of $\mathcal{L}$ is  independent of $\gamma$.  

\subsection{The EL-labelling}\label{sec: EL}
We will recall from \cite{ABW07} that the noncrossing partition lattice $\mathcal{L}$ has an EL-labelling. This leads to some important topological and combinatorial consequences, which are summarised in \propref{prop: dimtop} and \propref{prop: dimtoprksel}.

The set of covering relations  $\mathcal{E}(\mathcal{L})$ consists of the pairs $(u,v)$ with $u^{-1}v\in T$ for $u\lessdot v\leq \gamma$.  This defines a natural  edge labelling 
\begin{equation*}\label{eq: edge-labelling}
  \lambda: \mathcal{E}(\mathcal{L}) \rightarrow T, \quad (u,v) \mapsto u^{-1}v. 
\end{equation*}
Let  ${\bf c}: x=w_0< w_1< \dots < w_k=y$ be a maximal chain of any closed interval $[x,y]$ of $\mathcal{L}$. In the closed interval $[x,y]$ we may identify ${\bf c}$ with its  labelling sequence $\lambda({\bf c}):=(w_0^{-1}w_1, \dots, w_{k-1}^{-1}w_k)$, where $w_{i-1}^{-1}w_i\in T$ for $1\leq i\leq k$.

We may choose a total order $\preceq$ on $T$ which is irrelevant to the absolute order $\leq$. The chain ${\bf c}$ is said to be increasing (resp. decreasing) if $\lambda({\bf c})$ is strictly increasing (resp. decreasing) with respect to $\preceq$ on $T$. Any two maximal chains ${\bf c}$ and ${\bf c}'$ of $[x,y]$ can be compared by the lexicographical order induced by $\preceq$ on  their  labelling sequences.

It has been proved that there exists a total order $\preceq$ on $T$ which turns $\lambda$ into an EL-labelling \cite{ABW07}.  Recall from \cite{Bjo} that an edge labelling $\lambda: \mathcal{E}(\mathcal{L})\rightarrow T$ is called an EL-labelling of $\mathcal{L}$  if for every  interval $[x,y]$ of $\mathcal{L}$ 
   \begin{enumerate}
     \item there is a unique increasing  maximal chain in $[x,y]$,  and
     \item this chain is lexicographically smallest among all maximal chains in $[x,y]$.
   \end{enumerate}
If  $\mathcal{L}$ has an EL-labelling, then it is  Cohen-Macaulay \cite[Theorem 2.3]{Bjo}, i.e. for any $x<y$ of $\mathcal{L}$ we have
\[
  \widetilde{H}_i(x,y)=0, \quad \forall i \neq \ell_T(y)-\ell_T(x)-2. 
\]
This total order $\preceq$ on $T$ can be given explicitly in terms of positive roots of $W$  as we now describe. 

Let $\Phi$ be the root system of $W$ with the  set $\Pi=\{\alpha_i \mid  i \in [n]\}$ of simple roots. Without loss of generality, we may assume that $W$ is irreducible and hence the Coxeter graph  of $W$ is a tree, which is a bipartite graph. Then  $\Pi_1$ and $\Pi_2$ can be chosen such that they are associated with vertices of  two parts of the graph, respectively. Precisely,  $\Pi$ can be written as the disjoint union $\Pi=\Pi_1 \cup \Pi_2$, 
where $\Pi_1=\{\alpha_{i_1},\dots, \alpha_{i_l}\}$ and $\Pi_2=\{\alpha_{i_l+1}, \dots, \alpha_{i_n}\}$ for which $\alpha_{i_k}\in \Pi_1$ are mutually orthogonal as also are $\alpha_{i_k}\in \Pi_2$ (see \cite{Ste59}).

The set of positive roots is in bijection with the set of reflections of $W$. Recall that $W$ acts faithfully on the Euclidean space $\mathbb{R}^n$ whose inner product we denote by $(-,-)$. For any positive root $\alpha$ relative to $\Pi$, the corresponding reflection  is defined by
$t_{\alpha}(x):= x - 2 
\frac{(\alpha,x)}{(\alpha,\alpha)}\alpha$ for any $x\in \mathbb{R}^n$.
Throughout, we use the following  Coxeter element
\begin{equation}\label{eq:Coxelmt}
     \gamma= (\prod_{\alpha\in \Pi_1} t_{\alpha})(\prod_{\alpha\in \Pi_2} t_{\alpha}).
\end{equation}
Note that the simple reflections $t_{\alpha}$ and $t_{\beta}$ commute whenever $\alpha, \beta\in \Pi_1$ or $\alpha, \beta\in \Pi_2$.

Now we define a total order on the set of positive roots. Let $h$ be the Coxeter number, i.e. the order of $\gamma$. Then the number of positive roots is $nh/2$. It is proved in \cite[Theorem 6.3]{Ste59} that the positive roots $\rho_k$ of $\Phi$ relative to $\Pi$ can be produced successively using the following formulae 
\begin{equation}\label{eq: posroot}
   \rho_k= \begin{cases}
  \alpha_{i_k}, \quad  & 1\leq k\leq l,\\
  -\gamma(\alpha_{i_k}), \quad &l+1\leq k\leq n,\\
  \gamma(\rho_{k-n}), \quad &n+1\leq k\leq  \frac{nh}{2}. 
 \end{cases}
\end{equation}
 This yields a total order $\preceq$ on the set $T$ of reflections
\begin{equation}\label{eq: total ordering}
  t_{\rho_1}\prec t_{\rho_2}\prec \dots \prec t_{\rho_{nh/2}}.
\end{equation}
This total order will be used throughout the paper.

\begin{theorem}\label{thm: edgelab}\cite[Theorem 4.2]{ABW07}
 If the set $T$ of reflections is totally ordered by \eqref{eq: total ordering}, then the natural edge labelling
  \[  \lambda: \mathcal{E}(\mathcal{L}) \rightarrow T, \quad (u,v) \mapsto u^{-1}v. \]
   is an EL-labelling of the lattice $\mathcal{L}$.
\end{theorem}

\begin{example}\label{exam: dih}
 (Dihedral group) Let $W$ be the Coxeter group of type $I_2(m)$ defined by 
 \[
  W:=\langle s_1, s_2 \mid s_1^2=s_2^2=(s_2s_1)^m=1\rangle, \quad m\geq 3, 
 \]
 where $s_1, s_2$ are  simple reflections associated to simple roots $\alpha_1,\alpha_2$ of $W$, respectively.
 We choose the Coxeter element $\gamma=s_1s_2$. Using \eqref{eq: posroot}, we obtain reflections  $t_i= s_1(s_2s_1)^{i-1}, 1\leq i\leq m$  associated to all successive positive roots. Therefore, the set of reflections is totally ordered by
 \begin{equation}\label{eq: totord}
 T=\{t_1(=s_{1})\prec t_2\prec \dots \prec  t_{m-1}\prec t_{m}(=s_{2})\}.
 \end{equation} 
In the associated NCP lattice $\mathcal{L}=[e,\gamma]$, there are $m$ maximal chains corresponding to the $T$-reduced factorisations $\gamma=t_1t_m=t_2t_1=t_3t_2=\dots =t_{m}t_{m-1}$. Note that, with respect to the total ordering of $T$, $\gamma=t_1t_m=s_1s_2$ is the unique increasing factorisation and all other factorisations are decreasing.
\end{example}

\begin{example}\label{exam: A3}
	(Type $A_3$) Let $W={\rm Sym}_4$ be the symmetric group which acts on   $\mathbb{R}^4$ by permuting the standard basis  $e_i, 1\leq i\leq 4$. 
	Let $\Phi^+=\{e_i-e_j\mid 1\leq i<j\leq 4\}$ be the set of positive roots with respect to the simple  system $\Pi=\{\alpha_i:= e_i-e_{i+1}\mid 1\leq i\leq 3\}$. The reflection associated to the positive root $e_i-e_j$ is $(ij)$ in the cycle notation. We partition the simple system as  $\Pi=\Pi_1\cup \Pi_2=\{\alpha_1,\alpha_3\}\cup \{\alpha_2\}$, and choose the Coxeter element $\gamma=t_{\alpha_1}t_{\alpha_3}t_{\alpha_2}=(1243)$, which is of the form given in \eqnref{eq:Coxelmt}. According to \eqnref{eq: posroot}, the positive roots and reflections are totally ordered as follows:
	\begin{enumerate}
		\item Positive roots: $\{\alpha_1\prec\alpha_3\prec\alpha_1+\alpha_2+\alpha_3\prec\alpha_2+\alpha_3\prec\alpha_1+\alpha_2\prec\alpha_2 \}$; 
		\item Reflections: $\{(12)\prec (34)\prec(14)\prec(24)\prec(13)\prec(23)\}$.
	\end{enumerate}
 \figref{fig: A3labelling} shows  the edge labelling of $\mathcal{L}=[(1), (1243)]$ given by \thmref{thm: edgelab}. The unique increasing maximal  chain is labelled by the sequence $((12), (34), (23))$. There are five decreasing chains  labelled by the following sequences:
	\[ 
	\begin{aligned}
		&((14), (34), (12)),\quad  ((13), (14), (12)), \quad ((24), (14), (34)),  \\
		&((13), (24), (14)),\quad  ((23), (13), (24)).
	\end{aligned}
	\]
\end{example}

\begin{figure}[h]
	
	\begin{tikzpicture}[scale=0.8]
		
		\node [scale=0.8] at (7,-0.6) {(1)};
		
		\foreach \x\y in  {0/(12),2.8/(13),5.6/(14),8.4/(23),11.2/(24),14/(34)} {  \node [scale=0.8] at (\x ,2) {\y};}
		\foreach \x\y in  {0/(123),2.8/(124),5.6/(143),8.4/(243),11.2/(12)(34),14/(13)(24)} {\node [scale=0.8] at (\x ,6) {\y};}
		
		\node [scale=0.8] at (7,8.5) {(1243)};
		
		\foreach \x\y in {0/(12),2.8/(13),5.6/(14),8.4/(23),11.2/(24),14/(34)}  { \draw [line width=0.25mm](7,-0.3) -- (\x,1.8) node [scale=0.6] [midway,fill=white] {\y};}
		\foreach \x\y in {0/(24),2.8/(34),5.6/(12),8.4/(13),11.2/(23),14/(14) } { \draw [line width=0.25mm](7,8.3) -- (\x,6.2) node [scale=0.6] [midway,fill=white] {\y};}
		
		\foreach \x\y in {0/(23),2.8/(24),11.2/(34)} {\draw [line width=0.25mm] (0,2.2) -- (\x,5.8)  node [scale=0.6] [near start,fill=white] {\y};}
		\foreach \x\y in {5.6/(14),14/(24)} {\draw [line width=0.25mm] (2.8,2.2) -- (\x,5.8) node [scale=0.6] [very near start,fill=white] {\y};}
		\foreach \x\y in {0/(12)} {\draw [line width=0.25mm] (2.8,2.2) -- (\x,5.8) node [scale=0.6] [near end,fill=white] {\y};}

		\foreach \x\y in {2.8/(12),5.6/(34)} {\draw [line width=0.25mm] (5.6,2.2) -- (\x,5.8) node [scale=0.6] [near end,fill=white] {\y};}
		\foreach \x\y in {0/(13),8.4/(24)} {\draw [line width=0.25mm] (8.4,2.2) -- (\x,5.8) node [scale=0.6] [very near start,fill=white] {\y};}
		\foreach \x\y in {2.8/(14),8.4/(34)} {\draw [line width=0.25mm] (11.2,2.2) -- (\x,5.8) node [scale=0.6] [near start,fill=white] {\y};}
		
		\foreach \x\y in {14/(13)} {\draw [line width=0.25mm] (11.2,2.2) -- (\x,5.8) node [scale=0.6] [near end,fill=white] {\y};}
		
		\foreach \x\y in {5.6/(13)} {\draw [line width=0.25mm] (14,2.2) -- (\x,5.8) node [scale=0.6] [very near end,fill=white] {\y};}
		\foreach \x\y in {8.4/(23)} {\draw [line width=0.25mm] (14,2.2) -- (\x,5.8) node [scale=0.6] [midway,fill=white] {\y};}
		\foreach \x\y in {11.2/(12)} {\draw [line width=0.25mm] (14,2.2) -- (\x,5.8) node [scale=0.6] [near start,fill=white] {\y};}

	\end{tikzpicture}
	\caption{The $EL$-labelling of $\mathcal{L}$ with $\gamma=(1243)$}
	\label{fig: A3labelling}
\end{figure}
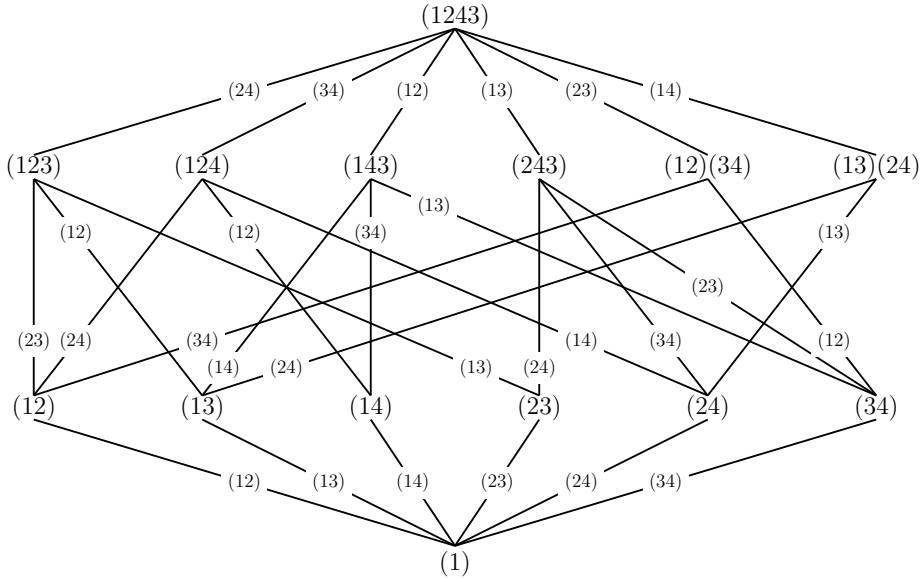

If $W$ is a finite Coxeter group, then $\mathcal{L}$ is a direct product of the NCP lattices $\mathcal{L}(W_i)$ over irreducible components $W_i$ of $W$. It is a result of \cite{Bjo} that the EL-labelling is preserved under the direct product. Moreover, any closed interval $[e,w]$ of $\mathcal{L}$ has an EL-labelling,  which is given by the restriction $\lambda|_{\mathcal{E}([e,w])}$ of the natural labelling $\lambda: \mathcal{E}(\mathcal{L}) \rightarrow T$.

\subsection{Combinatorial results} Recall that if $\mathcal{L}$ has an EL-labelling then $\mathcal{L}$ is  Cohen-Macaulay. This fact relates the M\"{o}bius function $\mu$ of $\mathcal{L}$ to the top homology groups of intervals of $\mathcal{L}$, and gives a combinatorial way to compute $\mu$ \cite{Sta12}.

For any $w\in \mathcal{L}$ we define 
 \begin{equation}\label{eq: Xik-1}
     \mathcal{D}_{w}:=\{ (t_1,\dots, t_{k})\mid  \text{$w=t_1\cdots t_k$ is $T$-reduced and \,} t_1\succ t_2\succ\dots\succ t_k \}.
 \end{equation}
 In words, $\mathcal{D}_{w}$ is the set of decreasing labelling sequences for the maximal chains $e<t_1<t_1t_2<\dots <t_1t_2\cdots t_k=w$ of $[e,w]$. Note that any interval $[u,v]$ of $\mathcal{L}$ is isomorphic to $[e, u^{-1}v]$ as posets \cite{ABW07}. 

\begin{proposition}\label{prop: dimtop}
 For any $w\in \mathcal{L}$ with $1\leq \ell_{T}(w)=k\leq n$, we have
\[{\rm rank}\, \widetilde{H}_{k-2}(e,w)= (-1)^{k}\mu(w)=|\mathcal{D}_w|. \]
\end{proposition}
\begin{proof}
  Each closed interval $[e,w]$ of $\mathcal{L}$ is isomorphic to a direct product of irreducible noncrossing partition lattices of smaller ranks (see \cite[Lemma 1.4.3]{Bes03} and \cite[Proposition 2.6.11]{Arm09}). Hence we may assume $w=\gamma$. Since $\mathcal{L}=[e,\gamma]$ is Cohen-Macaulay, the first equation follows from \cite[Proposition 3.8.6]{Sta12}, while the second equation is a consequence of \cite[Theorem 3.14.2]{Sta12}.
\end{proof}

Similarly, for each $k=2, \dots, n$ we define 
\begin{equation}\label{eq: Dk}
     \mathcal{D}_{[k-1]}:=\Big\{ (t_1,\dots, t_{k}) \Bigm|
     \begin{matrix}
       &  \text{\, $\gamma=t_1t_2\cdots t_n$ is $T$-reduced and \, } \\
       &t_1\succ \dots \succ t_{k-1}\succ t_{k}\prec t_{k+1}\prec \dots\prec t_{n}
     \end{matrix}
      \Big\}, 
\end{equation}
that is,  $\mathcal{D}_{[k-1]}$ is a collection of the first $k$ entries of the labelling sequences $(t_1,t_2, \dots, t_n)$ of the maximal chain $e<t_1<t_1t_2< \dots < t_1t_2\cdots t_n$ such that $(t_1,t_2, \dots, t_k)$ is decreasing, and $(t_k,t_{k+1}, \dots, t_n)$ is the labelling of the unique increasing maximal chain of the interval $[t_1t_2,\dots t_{k-1}, \gamma]$. Here the uniqueness follows from the definition of the EL-labelling.

\begin{proposition}\label{prop: dimtoprksel}
 For each $k=1, \dots , n-1$, let $\mathcal{L}_{[k]}=\{w\in \mathcal{L}\mid 1\leq \ell_{T}(w)\leq k\}$ be the rank-selected subposet of $\mathcal{L}$. Then 
   \[ {\rm rank}\, \widetilde{H}_{k-1}(\mathcal{L}_{[k]})=(-1)^{k+1}\mu(\hat{\mathcal{L}}_{[k]})=|\mathcal{D}_{[k]}|.  \]
\end{proposition}
\begin{proof}
It is a result of \cite[Theorem 6.4]{Bac80} that $\mathcal{L}_{[k]}$ is a Cohen-Macaulay poset. Hence the proposition follows from the combination of \cite[Proposition 3.8.6]{Sta12} and  \cite[Theorem 3.14.2]{Sta12}.
\end{proof}

From a topological point of view,  the order complex $\Delta(e,w)$ of any interval $(e,w)$ of $\mathcal{L}$ has the homotopy type of a bouquet of spheres, so does $\Delta(\mathcal{L}_{[k]})$ for each $k$ (see \cite[Theorem 2.3]{Bjo} and \cite[Theorem 1.3]{Bjo84}).

\section{Construction of the top homology cycles}\label{sec: cons}

In this section we give an explicit construction of the top homology cycles of  $\widetilde{H}_{\ell_T(w)-2}(e,w)$.  To this end, we make extensive use of the Hurwitz action on the chains of the lattice $\mathcal{L}$. 

\subsection{The Hurwitz action}

For each $k=0, \dots, n$, let  $ \mathcal{L}_k:=\{ w\in \mathcal{L} \mid \ell_{T}(w)=k \}$. 
For any $w\in \mathcal{L}_k$, define the set
\[ {\rm Rex}_{T}(w):=\{(t_1,t_2,\dots, t_k)\,|\, w=t_1t_2\cdots t_k \text{\, is $T$-reduced} \}.\]
In particular, ${\rm Rex}_{T}(e):=\emptyset$, the empty set.
The set ${\rm Rex}_{T}(w)$ is in bijection with the set of maximal chains  in the interval $[e,w]$ of $\mathcal{L}$, as 
each sequence $(t_1,t_2,\dots, t_k)$ may be thought of as the edge labelling of the  maximal chain $e<t_1<t_1t_2< \dots< t_1t_2\cdots t_k$ in $[e,w]$.

For any $w\in \mathcal{L}_{k}$, let $C_{k-1}(w)$ be the abelian group freely spanned by all  sequences $(t_1,\dots ,t_k)$ of ${\rm Rex}_T(w)$. In particular, $C_{-1}(e):=\mathbb{Z}$.  Let $B_k$ denote the braid group of $k$ strings with standard generators $\sigma_i, 1\leq i\leq k-1$ subject to the following relations
\begin{equation*}\label{eq: braidrel}
  \begin{aligned}
  \sigma_i\sigma_j&=\sigma_j\sigma_i, &\quad& |i-j|\geq 2,\\
  \sigma_i\sigma_{i+1}\sigma_i&=\sigma_{i+1}\sigma_i\sigma_{i+1}, &\quad& 1\leq i\leq k-2.
\end{aligned}
\end{equation*}
For each integer $k\geq 2$, the Hurwitz action of $B_k$ on $C_{k-1}(w)$ \cite{Bes03} is defined  by 
\begin{equation}\label{eq: Braidact}
    \sigma_i.(t_1, \dots,t_{i-1}, t_i, t_{i+1},t_{i+2}, \dots, t_k):=(t_1, \dots,t_{i-1}, t_{i+1}, t_{i}^{t_{i+1}}, t_{i+2},\dots, t_k)
\end{equation}
for $1\leq i\leq k-1$. If $k=1$, the trivial braid group $B_1$ acts trivially on $C_0(t)$ for any $t\in T$.

We introduce a useful $\mathbb{Z}$-linear transformation on $C_{k-1}(w)$ for $k\geq 2$. Let ${\rm Sym}_k$ be the symmetric group on $k$ letters with standard generators $s_i=(i,i+1), 1\leq i\leq k-1$. There is  a set-theoretic lift map:
\begin{equation*}
   \varphi:  {\rm Sym}_k \rightarrow B_k,\quad \pi=s_{i_1} \cdots s_{i_p}\mapsto \underline{\pi}:= \sigma_{i_1}\cdots \sigma_{i_p},
\end{equation*}
  where $\pi=s_{i_1} \cdots s_{i_p}\in {\rm Sym}_k$ is a reduced expression in standard generators. This is independent of the choice of reduced expression of $\pi$. 
 For any ${\bf t}=(t_1,t_2,\dots,t_k)\in {\rm Rex}_{T}(w)$, we  define the following $\mathbb{Z}$-linear map using the Hurwitz action \eqref{eq: Braidact}: 
 \begin{equation*}\label{eq: beta}
    \beta: C_{k-1}(w) \rightarrow C_{k-1}(w),\quad {\bf t}\mapsto \beta_{{\bf t}}:=\sum_{\pi\in {\rm Sym}_k} {\rm sgn}(\pi)\, \underline{\pi}.(t_1,t_2,\dots,t_k),
 \end{equation*} 
where ${\rm sgn}$ is the usual sign character  of ${\rm Sym}_k$. This can be   viewed as a  $\mathbb{Z}$-linear map on the  abelian group freely spanned by maximal chains of the  interval $[e,w]$.  As a convention, we define $\beta_{t}=\beta_{(t)}:=(t)$ for any $t\in T$ and $\beta_{\emptyset}=1$.

For any two sequences ${\bf t}=(t_1, \dots, t_k)$ and ${\bf t}=(t_1^{\prime}, \dots, t_{l}^{\prime})$, we write $({\bf t}, {\bf t}^{\prime})$ for the concatenation of $\bf{t}$ and ${\bf t}^{\prime}$. We extend the  concatenation $\mathbb{Z}$-bilinearly, i.e. $(\sum_{{\bf t}}\lambda_{{\bf t}} {\bf t}, \sum_{{\bf t}'}\mu_{{\bf t}'} {\bf t}')= \sum_{{\bf t}, {\bf t}'}\lambda_{{\bf t} } \mu_{{\bf t}'}({\bf t},{\bf t}')$ for $\lambda_{{\bf t} },  \mu_{{\bf t}'}\in \mathbb{Z}$. 
 
 The following lemma will be useful for the purpose of induction. 

\begin{lemma}\label{lem: betasum} 
Let ${\bf t}=(t_1,t_2,\dots,t_k)\in {\rm Rex}_{T}(w)$ with $w\leq \gamma$ and $2\leq k\leq n$. Denote by
\[ {\bf t}(\hat{i}):=(t_1,\dots,t_{i-1},t_{i+1},\dots,t_{k})\]
  the sequence obtained by removing  the $i$-th term of ${\bf t}$. 
 \begin{enumerate}
   \item  The sequence ${\bf t}(\hat{i})$ is the edge labelling of the following chain in the lattice $\mathcal{L}$:
     \[ t_1< t_1t_2<\dots <t_1\cdots t_{i-1} < t_1\cdots t_{i-1}t_{i+1} < \dots < t_1\cdots t_{i-1}t_{i+1}\cdots t_k. \]
     
   \item  $\beta_{{\bf t}}=\sum_{i=1}^k (-1)^{k-i} (\beta_{{\bf t}(\hat{i})},  \, t_i^{t_{i+1}\cdots t_k} )$;
   \item  $\beta_{{\bf t}}=\sum_{i=1}^{k} (-1)^{i-1}(t_i, \beta_{(t_1^{t_i},\dots, t_{i-1}^{t_i},\, t_{i+1},\dots,t_k)})$. 
 \end{enumerate}

\end{lemma}

\begin{proof}
  For part (1),  using the Hurwitz action \eqref{eq: Braidact} we have
 \begin{equation}\label{eq: seqHact}
  \sigma_{k-1}\dots\sigma_{i}. {\bf t}=(t_1,\dots,t_{i-1},t_{i+1},\dots,t_{k}, t_i^{t_{i+1}\cdots t_k})=({\bf t}(\hat{i}), t_i^{t_{i+1}\cdots t_k}),
\end{equation} 
 Since the  Hurwitz action preserves  $T$-reduced expressions, the sequence ${\bf t}(\hat{i})$ is the labelling for the chain of $\mathcal{L}$ given in part (1).

  We turn to prove part (2).  Let ${\rm Sym}_{k-1}$ be the subgroup of ${\rm Sym}_k$ generated by $s_i=(i,i+1), 1\leq i \leq k-2$. Then the  minimal right coset representatives of ${\rm Sym}_{k-1}$ in ${\rm Sym}_k$ are  $\zeta_i=s_{k-1}s_{k-2}\cdots s_{i}, 1\leq i \leq k-1$ and  $\zeta_k=(1)$. Therefore we may write each $\pi \in {\rm Sym}_k$ uniquely as $\pi=\pi^{\prime} \zeta_{i}$ for some $i$, where $\pi^{\prime}\in {\rm Sym}_{k-1}$ and ${\rm sgn}(\pi)=(-1)^{k-i}{\rm sgn}(\pi^{\prime})$. Notice that  the lift of $\zeta_i$ in $B_k$ is $\underline{\zeta}_i=\sigma_{k-1}\cdots \sigma_i$.  Using \eqref{eq: seqHact} we have 
   \[    \underline{\pi}.{\bf t}= \underline{\pi}^{\prime} \sigma_{k-1}\cdots \sigma_i.{\bf t}
       =\underline{\pi}^{\prime}.({\bf t}(\hat{i}),\, t_i^{t_{i+1}\cdots t_k} ),   \]
  where $\underline{\pi}^{\prime}\in B_{k-1}$ only acts on the sequence ${\bf t}(\hat{i})$.
  It follows that 
   \begin{align*}
     \beta_{{\bf t}} &= \sum_{\pi\in {\rm Sym}_k} {\rm sgn}(\pi) \underline{\pi}.{\bf t}=\sum_{i=1}^{k} \sum_{\pi^{\prime}\in {\rm Sym}_{k-1}} (-1)^{k-i}{\rm sgn}(\pi^{\prime}) \underline{\pi}^{\prime}. ({\bf t}(\hat{i}),\, t_i^{t_{i+1}\cdots t_k} )\\
   &=\sum_{i=1}^k (-1)^{k-i} (\beta_{{\bf t}(\hat{i})},\, t_i^{t_{i+1}\cdots t_k})
   \end{align*}
as required. Part (3) can be proved similarly.
\end{proof}

\subsection{Construction of the homology cycle}
We will construct explicitly the  cycles of the top homology groups of intervals in $\mathcal{L}$. For each  $k=0,1,\dots, n$, we define  
\begin{equation}\label{eq: chaingp}
	C_{k-1}=C_{k-1}(\mathcal{L}):=\bigoplus_{w\in \mathcal{L}_k} C_{k-1}(w).
\end{equation}
Then $C_{k-1}$ is  a free  abelian group with a $\mathbb{Z}$-basis $\bigcup_{w\in \mathcal{L}_k} {\rm Rex}_T(w)$. Define $\mathcal{C}:=\bigoplus_{k=0}^n C_{k-1}$.

The following free abelian group  will play an important role in the construction.

\begin{definition}\label{def: AlgB}
  For each $w\in \mathcal{L}_k$ with $0\leq k\leq n$, we define $\mathcal{B}_w$ to be the  subgroup of $C_{k-1}(w)$ spanned by all  $\beta_{{\bf t}}, {\bf t} \in {\rm Rex}_{T}(w)$, that is, 
 \[ \mathcal{B}_{w}:={\rm Im}\, \beta=\sum_{{\bf t}\in {\rm Rex}_{T}(w)} \mathbb{Z}\beta_{{\bf t}} \subseteq C_{k-1}(w).\]
In particular $\mathcal{B}_e:=\mathbb{Z}$.  We define $\mathcal{B}:=\bigoplus_{k=0}^n \bigoplus_{w\in \mathcal{L}_k}\mathcal{B}_w$ and $\mathcal{B}_k:= \bigoplus_{w\in \mathcal{L}_k}\mathcal{B}_w (0\leq k\leq n)$, which are free abelian subgroups of $\mathcal{C}$ and $C_{k-1} (0\leq k\leq n)$, respectively.
\end{definition}

In general,  the  elements $\beta_{{\bf t}}, {\bf t}\in {\rm Rex}_{T}(w)$  are  not $\mathbb{Z}$-linearly independent. We will show in \thmref{thm: homint} that some of them form a basis for $\mathcal{B}_w$.

We proceed next to study the connection between $\mathcal{B}$ and the homology of the lattice $\mathcal{L}$. 
For any $w\in \mathcal{L}_k$, let 
$\widetilde{H}_{k-2}(e,w):=\widetilde{H}_{k-2}(\Delta(e,w); \mathbb{Z})$
be the top reduced homology group of the open interval $(e,w)$. For each integer $k=2,\dots, n$, define the $\mathbb{Z}$-linear truncation  $d_{k-1}$ by
\begin{equation}\label{eq: theta}
 d_{k-1}: C_{k-1}\rightarrow C_{k-2}, \quad (t_1,t_2,\dots,t_{k-1},t_k)\mapsto (t_1,t_2,\dots, t_{k-1}),  
\end{equation}
and  $d_0(t)=1, \forall t\in T$. We will write $d=d_{k-1}$ whenever no confusion arises.


\begin{lemma}\label{lem: restheta}
  For any $w\in \mathcal{L}_{k}$ with $k\geq 1$, the restriction $d|_{C_{k-1}(w)}$ is injective.
 \end{lemma}
 \begin{proof}
Assume  there exist integers  $\lambda_{(t_1, \dots, t_k)}$ such that
\[
\begin{aligned}
0&=\sum_{(t_1, \dots, t_k)\in {\rm Rex}_T(w)}d(\lambda_{(t_1, \dots, t_k)}(t_1, \dots, t_k))=\sum_{(t_1, \dots, t_k)\in {\rm Rex}_T(w)}\lambda_{(t_1, \dots, t_k)}(t_1, \dots, t_{k-1}).\\
\end{aligned}
 \]
Note that  each sequence $(t_1, \dots , t_{k-1})$ can be uniquely extended to  $(t_1, \dots, t_{k-1}, t_k)$, where   $t_k= t_{k-1}\cdots t_1 w$.
It follows that the elements   $(t_1, \dots, t_{k-1})$  on the far right hand side are distinct elements of the basis for $C_{k-2}$, whence all coefficients $\lambda_{(t_1, \dots, t_k)}$ are equal to $0$. Therefore,  the restriction $d|_{C_{k-1}(w)}$ is injective.
 \end{proof}
 
 Since $\mathcal{B}_w\subseteq C_{k-1}(w)$, we may restrict $d$ to  $\mathcal{B}_w$ and  define the element
 \begin{equation*}\label{eq: defz}
   z_{{\bf t}}:=d(\beta_{\bf t}), \quad \forall {\bf t}\in {\rm Rex}_T(w).
 \end{equation*}
 We have the following important observations.

 \begin{proposition}\label{prop: cycle}
  Let $w\in \mathcal{L}_{k}$ with $k\geq 1$. 
   \begin{enumerate}
     \item  We have $z_{{\bf t}}\in \widetilde{H}_{k-2}(e,w)$ for any ${\bf t}\in {\rm Rex}_{T}(w)$.
     \item The $\mathbb{Z}$-linear map 
     \begin{equation}\label{eq: linearBH}
       d: \mathcal{B}_w \rightarrow \widetilde{H}_{k-2}(e,w), \quad \beta_{{\bf t}} \mapsto z_{{\bf t}}, \, {\bf t}\in {\rm Rex}_{T}(w)
     \end{equation}
         is injective.
 
     \item 
     We have the following expression:
     \begin{equation}\label{eq: zexp}
     z_{{\bf t}}=\sum_{i=1}^k (-1)^{k-i} \beta_{{\bf t}(\hat{i})},
   \end{equation}
   where ${\bf t}(\hat{i}):=(t_1,\dots, \hat{t}_i,\dots,t_k)$ is obtained by removing the $i$-th entry of ${\bf t}$.
   \end{enumerate}
 \end{proposition}

 Let us recall some facts before proving \propref{prop: cycle}. For any $w\in \mathcal{L}_{k}$,  all maximal chains in the open interval $(e,w)$ are of the form $t_1< t_1t_2< \dots < t_1t_2\cdots t_{k-1}$, which can be identified with the sequence $\bar{{\bf t}}:=(t_1,t_2,\dots, t_{k-1})$. These maximal chains correspond to the $(k-2)$-simplices in the order complex $\Delta(e,w)$ of dimension $k-2$.  Then we may rewrite the usual boundary operator on simplices as 
 \begin{equation*}\label{eq: boundary}
      \partial(\bar{{\bf t}})=\sum_{i=1}^{k-1}(-1)^{i-1}\partial^i(\bar{{\bf t}}), 
 \end{equation*}
 where $\partial^i(\bar{{\bf t}})$ amounts to removing the $i$-th term in the maximal chain $t_1< t_1t_2< \dots< t_1t_2\cdots t_{k-1}$ and hence is defined as 
 \begin{equation*}\label{eq: partial}
      \partial^i(\bar{{\bf t}}):=
    \begin{cases}
       (t_1,\dots,t_{i-1}, t_it_{i+1},t_{i+2},\dots, t_{k-1}), &\quad 1\leq i\leq k-2,\\
       (t_1,t_2,\dots,t_{k-2}), &\quad i=k-1. 
    \end{cases}
 \end{equation*}

 \begin{proof}[Proof of Proposition~\ref{prop: cycle}]
  For part (1), we  need to prove that  
  $$\partial(z_{{\bf t}})=\sum_{i=1}^{k-1}(-1)^{i-1} \partial^i(z_{{\bf t}})=0.$$
  By the definition of $\partial$  it suffices to show that  
  \begin{equation}\label{eq: parz}
    \partial^i(z_{{\bf t}})=\sum_{\pi \in {\rm Sym}_k} {\rm sgn}(\pi)\, \partial^i  d (\underline{\pi}.(t_1, t_2, \dots,t_{k}))=0, \quad  1\leq i\leq k-1.
  \end{equation}
  First we suppose that $i\leq k-2$. If $\pi \in {\rm Sym}_k$ let $\eta=(i,i+1)\pi $. Now two cases arise. If $(i,i+1)$ is a left descent of $\eta$, then $\underline{\eta}=\sigma_i \underline{\pi}$, where $\sigma_i, \underline{\pi}\in B_k$ are lifts of $(i,i+1)$ and $\pi$, respectively. Denoting $(t_1^{\pi},t_2^{\pi}, \dots ,t_{k}^{\pi}):=\underline{\pi}.(t_1, t_2, \dots,t_{k})$, we have 
  \[
    \begin{aligned}
         \partial^id  (\underline{\eta}.(t_1, t_2, \dots,t_{k}))&=\partial^id  (\sigma_i. (t_1^{\pi},t_2^{\pi}, \dots ,t_{k}^{\pi})),\\
          &= \partial^id (t_1^{\pi},\dots, t_{i-1}^{\pi}, t_{i+1}^{\pi}, t_{i+1}^{\pi}t_{i}^{\pi}t_{i+1}^{\pi}, t_{i+2}^{\pi} ,\dots , t_{k}^{\pi})\\
          &= (t_1^{\pi},\dots, t_{i-1}^{\pi}, t_{i}^{\pi}t_{i+1}^{\pi}, t_{i+2}^{\pi} ,\dots , t_{k-1}^{\pi})\\
          &= \partial^id  (\underline{\pi} .(t_1, t_2, \dots,t_{k})).
    \end{aligned}
  \]
If $(i,i+1)$ is not a left descent of $\eta$, then we swap the roles of $\eta$ and $\pi$ and obtain the same equation as above. Therefore, in both two cases the terms in \eqref{eq: parz} corresponding to $\eta$ and $\pi$ cancel, which implies $\partial^i(z_{{\bf t}})=0$ for all $i\leq k-2$.  The remaining case $i=k-1$  can be proved similarly. 

For part (2), using part (1) we obtain $d(\mathcal{B}_w)\subseteq \widetilde{H}_{k-2}(e,w)$, where the latter is the top homology group of $(e,w)$. As $\mathcal{B}_{w}\subseteq C_{k-1}(w)$,  by \lemref{lem: restheta} the restriction $d|_{\mathcal{B}_w}$ is injective. Part (3) is an immediate consequence of  part (2) of \lemref{lem: betasum}.
\end{proof}

\begin{remark}
  There is a different construction of homology cycles. For  $1\leq i\leq k-1$, by \eqref{eq: Braidact} we have 
  $  \sigma_i^{-1}.(t_1, \dots, t_i, t_{i+1}, \dots, t_k):=(t_1, \dots, t_{i+1}^{t_i}, t_{i},\dots, t_k)$.
  Then we may define a similar element $\beta'_{{\bf t}}$ by 
  \[ \beta'_{{\bf t}}:=\sum_{\pi\in {\rm Sym}_k} {\rm sgn}(\pi)\, \underline{\pi}^{-1}.(t_1,t_2,\dots,t_k),\quad \forall {\bf t}=(t_1,\dots, t_k)\in {\rm Rex}_{T}(w), \]
  where $\underline{\pi}^{-1}= \sigma_{k}^{-1}\cdots \sigma_{1}^{-1}$ for a reduced expression $\pi=s_1\cdots s_k$.
Similarly, we have
  \[z'_{{\bf t}}=d(\beta'_{{\bf t}})=\sum_{i=1}^k (-1)^{k-i} \beta'_{(t_1,\dots,t_{i-1}, t_{i+1}^{t_i}, \dots, t_{k}^{t_i}) } \in \widetilde{H}_{k-2}(e,w).  \]
 Comparing with \eqref{eq: zexp}, the summand on the right hand side of the above equation involves conjugations.  We prefer to use the cycles $z_{{\bf t}}$ rather than $z'_{{\bf t}}$ in this paper. 
\end{remark}

\section{Bases for the top homology groups}\label{sec: basis}

In this section we give a basis for the top homology group $\widetilde{H}_{\ell_{T}-2}(e,w)$, and for the top homology group $\widetilde{H}_{k-1}(\mathcal{L}_{[k]})$ with $1\leq k\leq n-1$. 

\subsection{Basis for the top homology group of an interval}
We assume $[e,w]$ is a closed interval of the NCP lattice $\mathcal{L}$. We will give a $\mathbb{Z}$-basis for the top homotopy group $\widetilde{H}_{k-2}(e,w)$, where $k=\ell_{T}(w)$.  To this end, we will first obtain  a $\mathbb{Z}$-basis for $\mathcal{B}_w$, and then obtain a $\mathbb{Z}$-basis from it  for $\widetilde{H}_{k-2}(e,w)$ via the  the linear map $d$ given  in \eqref{eq: linearBH}.

 Recall from \secref{sec: EL} that the NCP lattice $\mathcal{L}$ admits an EL-labelling, with the total order $\preceq$ on $T$ given in \eqref{eq: total ordering}. For any $w\in \mathcal{L}$, let $\mathcal{D}_{w}$ be as defined in \eqnref{eq: Xik-1}.
 We agree that $\mathcal{D}_{e}=\{\emptyset\}$ and $\beta_{\emptyset}=1$. The elements $\beta_{{\bf t}}, {\bf t}\in \mathcal{D}_w$  are called \emph{decreasing elements}.

For any two sequences ${\bf t}=(t_1, t_2, \dots , t_k)$ and ${\bf t}'=(t'_1,t'_2,\dots , t'_k)$, we define the lexicographical order by 
\[
  {\bf t}\prec {\bf t}' \iff t_i\prec t'_i\quad  \text{for the smallest $i$ where $t_i$ and $t'_i$ differ}.
\]
Clearly, this is a total order on the set ${\rm Rex}_T(w)$. Recalling that $C_{k-1}(w)$ is the abelian group freely spanned by ${\rm Rex}_T(w)$, we define a symmetric $\mathbb{Z}$-bilinear form $\langle-,-\rangle$ on $C_{k-1}(w)$ by 
\begin{equation}\label{eq: biliform}
  \langle{\bf t}, {\bf t}'\rangle= \delta_{{\bf t}, {\bf t}'}, \quad  \forall {\bf t}, {\bf t}' \in {\rm Rex}_T(w),
\end{equation} 
and then extend this $\mathbb{Z}$-bilinearly. It is obvious that this bilinear form is unimodular, i.e., it defines an isomorphism between $C_{k-1}(w)$ and its dual.

\begin{lemma}\label{lem: maxterm}
 For any sequence ${\bf t}=(t_1,t_2, \dots, t_k)\in \mathcal{D}_w$, the sum $\beta_{{\bf t}}=\sum_{\pi\in {\rm Sym}_k} {\rm sgn}(\pi) \underline{\pi}.{\bf t}$ has a unique maximal term $(t_1,t_2, \dots, t_k)$ with respect to the lexicographical order $\prec$ on ${\rm Rex}_T(w)$. Similarly, the element $z_{{\bf t}}= d(\beta_{{\bf t}})$ has a unique maximal term $(t_1, \dots, t_{k-1})$. 
\end{lemma} 
\begin{proof}
Use induction on $k$. If $k=1$, then $\beta_{t}=(t)$ for any $t\in T$, so there is nothing to prove. 
By  \lemref{lem: betasum}, we have $\beta_{{\bf t}}=\sum_{i=1}^{k} (-1)^{i-1}(t_i, \beta_{(t_1^{t_i},\dots, t_{i-1}^{t_i},\, t_{i+1},\dots,t_k)})$. Since $t_1\succ t_2\succ \dots \succ t_k$, the maximal term of $\beta_{{\bf t}}$ begins with $t_1$ and hence appears in  $(t_1,\beta_{(t_2, \dots, t_k)})$. By induction hypothesis, the sum $(t_1, \beta_{(t_2, \dots, t_k)})$ has the unique maximal term $(t_1, t_2, \dots, t_k)$. Therefore, $(t_1, t_2, \dots, t_k)$ is the unique maximal term of the sum $\beta_{{\bf t}}$. 

Recall from \propref{prop: cycle} that $z_{{\bf t}}=d(\beta_{{\bf t}})= \sum_{i=1}^{k} (-1)^{k-i} \beta_{{\bf t}(\hat{i})}$. Note that ${\bf t}(\hat{i})=(t_1, \dots, \hat{t}_i, \dots, t_k)$ is a decreasing sequence. Hence each $\beta_{{\bf t}(\hat{i})}$ has a unique maximal term ${\bf t}(\hat{i})$. Among these ${\bf t}(\hat{i})$ the maximal one is ${\bf t}(\hat{k})=(t_1,\dots, t_{k-1})$, which is the unique maximal term of $z_{{\bf t}}$.
\end{proof}

The following result will be used frequently. 

\begin{lemma}\label{lem: bz}
Let $ \langle -,- \rangle $ be the bilinear form on $C_{k-1}(w)$ as defined in \eqref{eq: biliform}, where $k=\ell_{T}(w)$. 
\begin{enumerate}
  \item $\langle \beta_{\bf t}, {\bf t}\rangle =\langle z_{{\bf t}}, {\bf t}(\hat{k})  \rangle = 1$ for any ${\bf t}\in \mathcal{D}_w$.
  \item  $\langle z_{{\bf t}}, \tilde{{\bf t}}(\hat{k}) \rangle=0$ for any pair  ${\bf t}\prec \tilde{{\bf t}}$ of $\mathcal{D}_w$. 
\end{enumerate} 
\end{lemma} 
\begin{proof}
It follows from the proof of  \lemref{lem: maxterm} that the coefficient of the maximal term of $\beta_{{\bf t}}$ is $1$, and similarly for $z_{{\bf t}}$. This implies part (1). For part (2), note that the maximal term of $z_{{\bf t}}$ is ${\bf t}(\hat{k})$. We claim that ${\bf t}(\hat{k})\prec \tilde{{\bf t}}(\hat{k})$, so all terms in the expression of $z_{{\bf t}}$ are strictly smaller than $\tilde{{\bf t}}(\hat{k})$. This implies that   $\langle z_{{\bf t}}, \tilde{{\bf t}}(\hat{k}) \rangle=0$. 

It remains to prove our claim. Assume for contradiction that ${\bf t}(\hat{k})\succeq \tilde{{\bf t}}(\hat{k})$.  If ${\bf t}(\hat{k})\succ \tilde{{\bf t}}(\hat{k})$, then we have ${\bf t}\succ \tilde{{\bf t}}$, which contradicts the condition  ${\bf t}\prec \tilde{{\bf t}}$. If ${\bf t}(\hat{k})= \tilde{{\bf t}}(\hat{k})$, then the $k$-th components of ${\bf t}$ and $\tilde{{\bf t}}$ are also equal, as both  ${\bf t}$ and $\tilde{{\bf t}}$ represent $T$-reduced factorisations of $w$. It follows that ${\bf t}= \tilde{{\bf t}}$, which is a contradiction. Therefore, we have the claim that ${\bf t}(\hat{k})\prec \tilde{{\bf t}}(\hat{k})$.
\end{proof} 

\begin{proposition}\label{prop: decindp}
  For any $w\in \mathcal{L}$, the decreasing elements $\beta_{{\bf t}}, {\bf t}\in \mathcal{D}_w$ are $\mathbb{Z}$-linearly independent. 
\end{proposition}
\begin{proof}
 Assume that we have a finite sum  $\sum_{{\bf t}} \lambda_{{\bf t}} \beta_{{\bf t}}=0$ for some nonzero $\lambda_{{\bf t}}\in \mathbb{Z}$ and ${\bf t}\in \mathcal{D}_w$. Using the lexicographical order, we may arrange the above equation into 
\[ \lambda_{\tilde{{\bf t}}} \beta_{\tilde{{\bf t}}}= -\sum_{{\bf t}\prec  \tilde{{\bf t}}} \lambda_{{\bf t}} \beta_{{\bf t}},  \]
where $\tilde{{\bf t}}$ is the unique maximal decreasing sequence appearing in the sum  $\sum_{{\bf t}} \lambda_{{\bf t}} \beta_{{\bf t}}$. By \lemref{lem: maxterm}, for any ${\bf t}\prec \tilde{{\bf t}}$, $\beta_{\bf t}$ has a unique maximal term ${\bf t}$. Hence $\langle \beta_{{\bf t}}, \tilde{{\bf t}}\rangle =0$ for any decreasing sequence ${\bf t}\prec \tilde{{\bf t}}$. It follows that 
\[ \lambda_{\tilde{{\bf t}}}= \langle \lambda_{\tilde{{\bf t}}} \beta_{\tilde{{\bf t}}}, \tilde{{\bf t}} \rangle = - \sum_{{\bf t}\prec  \tilde{{\bf t}}} \lambda_{{\bf t}} \langle \beta_{{\bf t}}, \tilde{{\bf t}}\rangle=0, \]
which contradicts our assumption. Therefore, the elements $\beta_{{\bf t}}, {\bf t}\in \mathcal{D}_w$ are $\mathbb{Z}$-linearly independent. 
\end{proof} 

\begin{corollary}\label{coro: rankBw}
  For any $w\in \mathcal{L}$ we have 
   \[
     {\rm rank}\, \mathcal{B}_w = |\mathcal{D}_w|= (-1)^{\ell_{T}(w)}\mu(w),
   \]
   where $\mu $ denotes the M\"obius function of $\mathcal{L}$.
\end{corollary} 
\begin{proof}
Using \propref{prop: dimtop}, we have ${\rm rank}\, \widetilde{H}_{k-2}(e,w)=|\mathcal{D}_w|= (-1)^{\ell_{T}(w)}\mu(w)$. Recall from \propref{prop: cycle} that the linear map $d: \mathcal{B}_w \rightarrow \widetilde{H}_{k-2}(e,w)$ is injective. This implies that
  ${\rm rank}\, \mathcal{B}_w\leq {\rm rank}\, \widetilde{H}_{k-2}(e,w)= (-1)^{\ell_{T}(w)}\mu(w)$.
  On the other hand, since  $\beta_{{\bf t}}, {\bf t}\in \mathcal{D}_w$ are $\mathbb{Z}$-linearly independent, we have ${\rm rank}\, \mathcal{B}_w\geq |\mathcal{D}_w|= (-1)^{\ell_{T}(w)}\mu(w)$. It follows that ${\rm rank}\, \mathcal{B}_w= |\mathcal{D}_w|= (-1)^{\ell_{T}(w)}\mu(w)$. 
\end{proof}

\begin{theorem}\label{thm: homint}
 For any $w\in \mathcal{L}$, let $\mathcal{B}_{w}$ and $\mathcal{D}_w$ be as defined in \defref{def: AlgB} and \eqref{eq: Xik-1}, respectively. 
 \begin{enumerate}
   \item The elements  $\beta_{{\bf t}}, {\bf t}\in \mathcal{D}_w$ constitute a $\mathbb{Z}$-basis for $ \mathcal{B}_w$;
   \item  The elements $z_{{\bf t}}, {\bf t}\in \mathcal{D}_w$ are a $\mathbb{Z}$-basis for $ \widetilde{H}_{k-2}(e,w)$, where $k=\ell_T(w)\geq 1$; 
   \item The $\mathbb{Z}$-linear map 
   \begin{equation}\label{eq: Bwiso}
     d: \mathcal{B}_w \rightarrow \widetilde{H}_{k-2}(e,w), \quad \beta_{{\bf t}} \mapsto z_{{\bf t}}, \, {\bf t}\in {\rm Rex}_{T}(w)
   \end{equation}
   is an isomorphism of free abelian groups.
 \end{enumerate}
\end{theorem} 
\begin{proof}
 By \propref{prop: decindp}, the elements $\beta_{{\bf t}}, {\bf t}\in \mathcal{D}_w$ are $\mathbb{Z}$-linearly independent.  
  For part (1), it remains to prove that every element of $\mathcal{B}_w$ is a $\mathbb{Z}$-linear combination of those elements. Let $\mathcal{B}'_{w}$ be the subgroup of $\mathcal{B}_{w}$  spanned by these elements. Then by \corref{coro: rankBw} $\mathcal{B}'_{w}$  is a subgroup of maximal rank. Therefore, every element of the quotient group $\mathcal{B}_{w}/ \mathcal{B}'_{w}$ has finite order. Taking an arbitrary nonzero $\beta\in \mathcal{B}_{w}$, there exists a nonzero integer $m$ such that
  \begin{equation}\label{eq: intbeta}
    m\beta = \sum_{{\bf t}\in \mathcal{D}_{w}} \lambda_{{\bf t}} \beta_{{\bf t}}\in \mathcal{B}'_{w}, \quad \lambda_{{\bf t}}\in \mathbb{Z}.
  \end{equation}
  We need to show that $m| \lambda_{{\bf t}}$ for any ${\bf t}\in \mathcal{D}_{w}$.

Suppose that $\tilde{{\bf t}}$ is the biggest ${\bf t}$ for which $m\nmid \lambda_{{\bf t}}$, that is, $m|\lambda_{{\bf t}}$ for all ${\bf t}\succ\tilde{{\bf t}}$ and $m\nmid \lambda_{\tilde{{\bf t}}}$. Then \eqref{eq: intbeta} can be written as
 \begin{equation*}
     m\beta -\sum_{{\bf t}\succ \tilde{{\bf t}}} \lambda_{{\bf t}}\beta_{{\bf t}}= \lambda_{\tilde{{\bf t}}}\beta_{\tilde{{\bf t}}}+ \sum_{{\bf t}\prec\tilde{{\bf t}}}\lambda_{{\bf t}} \beta_{{\bf t}}.
  \end{equation*} 
 Noting that $\langle \beta_{{\bf t}}, \tilde{{\bf t}} \rangle=0$ for any decreasing sequence ${\bf t}\prec \tilde{{\bf t}}$, we have 
 \[
 \langle  m\beta -\sum_{{\bf t}\succ \tilde{{\bf t}}} \lambda_{{\bf t}}\beta_{{\bf t}}, \tilde{{\bf t}} \rangle = \langle \lambda_{\tilde{{\bf t}}} \beta_{\tilde{{\bf t}}}, \tilde{{\bf t}} \rangle=  \lambda_{\tilde{{\bf t}}},   
 \] 
 The far left hand side is divisible by $m$, as $m|\lambda_{{\bf t}}$ for all ${\bf t}\succ\tilde{{\bf t}}$.   It follows that $m|\lambda_{\tilde{{\bf t}}}$, which contradicts our assumption for $\tilde{{\bf t}}$. Therefore,  we have $m| \lambda_{{\bf t}}$ for any ${\bf t}\in \mathcal{D}_{w}$ and 
 $$\beta=\sum_{{\bf t}\in \mathcal{D}_{w}} \frac{\lambda_{{\bf t}}}{m} \beta_{{\bf t}}\in \mathcal{B}'_{w}, \quad \frac{\lambda_{{\bf t}}}{m}\in \mathbb{Z}.$$
Hence $\mathcal{B}'_{w}=\mathcal{B}_{w}$ and $\{\beta_{{\bf t}}\mid {\bf t}\in \mathcal{D}_{w}\}$ is a basis for $\mathcal{B}_{w}$. 

We turn to prove part (2). As the $\mathbb{Z}$-linear map $d: \mathcal{B}_w \rightarrow \widetilde{H}_{k-2}(e,w)$ is injective by \propref{prop: cycle},  the elements $z_{{\bf t}}=d(\beta_{{\bf t}}), {\bf t}\in \mathcal{D}_{w}$ are $\mathbb{Z}$-linearly independent. We just need to show that every element of $ \widetilde{H}_{k-2}(e,w)$ is a $\mathbb{Z}$-linear combination of $z_{{\bf t}}, {\bf t}\in \mathcal{D}_{w}$. The proof is similar to that in part (1), using the fact that 
$\langle z_{{\bf t}}, {\bf t}(\hat{k})  \rangle = 1$ and $\langle z_{{\bf t}}, \tilde{{\bf t}}(\hat{k}) \rangle=0$ for any decreasing sequences ${\bf t}\prec \tilde{{\bf t}}$ of $\mathcal{D}_w$. The fact is proved in \lemref{lem: bz}.  

Now part (3) is a consequence of the combination of part (1) and part (2).
\end{proof}

\subsection{Basis for the top homology group of a rank-selected subposet}
Now we are concerned with the top reduced homology group 
$\widetilde{H}_{k-1}(\mathcal{L}_{[k]}):=\widetilde{H}_{k-1}(\Delta(\mathcal{L}_{[k]});\mathbb{Z})$ of the  rank-selected subposet $ \mathcal{L}_{[k]}=\{\, w\in \mathcal{L} \mid 1\leq \ell_{T}(w)\leq k\, \}$  for  $ 0\leq k\leq n-1$. The trivial case is $\mathcal{L}_{[0]}:=\emptyset$ and $\widetilde{H}_{-1}(\mathcal{L}_{[0]})=\mathbb{Z}$.

Recall from \defref{def: AlgB}  that $\mathcal{B}$ has a natural $\mathbb{Z}$-grading with the $k$-th homogeneous component given by $\mathcal{B}_k= \bigoplus_{w\in \mathcal{L}_k} \mathcal{B}_w$, where $\mathcal{L}_k$ consists of elements $w\in \mathcal{L}$ with $0\leq \ell_T(w)=k\leq n$.

\begin{lemma}
  For any ${\bf t}\in \bigcup_{w\in \mathcal{L}_{k}} {\rm Rex}_{T}(w)$, we have $z_{{\bf t}}=d(\beta_{{\bf t}})\in \widetilde{H}_{k-2}(\mathcal{L}_{[k-1]})$ . This defines the following $\mathbb{Z}$-linear map:
  \begin{equation*}
  d_k: \mathcal{B}_{k}\rightarrow \widetilde{H}_{k-2}(\mathcal{L}_{[k-1]}), \quad \beta_{{\bf t}} \mapsto z_{{\bf t}}, \quad \forall {\bf t}\in \bigcup_{w\in \mathcal{L}_{k}} {\rm Rex}_{T}(w). 
\end{equation*}
Hence we have $d(\mathcal{B}_k) \subseteq \widetilde{H}_{k-2}(\mathcal{L}_{[k-1]})\subseteq \mathcal{B}_{k-1}$ for $1\leq k\leq n$. 
\end{lemma}
\begin{proof}
Recall from \propref{prop: cycle} that $z_{{\bf t}}$ is an element of $\widetilde{H}_{k-2}(e,w)$ for any ${\bf t}\in  {\rm Rex}_{T}(w)$. Since the interval $(e,w)$ is contained in $\mathcal{L}_{[k-1]}$, the element $z_{{\bf t}}$ is also a cycle of the top homology group $\widetilde{H}_{k-2}(\mathcal{L}_{[k-1]})$.
\end{proof}

\begin{proposition}\label{prop: selind}
 Let $\mathcal{D}_{[k-1]}$ be as in \eqnref{eq: Dk}.  Then the elements   $z_{{\bf t}}, {\bf t}\in \mathcal{D}_{[k-1]}$  are $\mathbb{Z}$-linearly independent. 
\end{proposition}
\begin{proof}
  Assume that  $\sum_{{\bf t}} \lambda_{{\bf t}} z_{{\bf t}}=0$ for some nonzero $\lambda_{{\bf t}}\in \mathbb{Z}$ and ${\bf t}\in \mathcal{D}_w$. Let $\tilde{{\bf t}}$ be the maximal sequence  appearing in the sum. Then we have 
  \[ \lambda_{\tilde{{\bf t}}} z_{\tilde{{\bf t}}}= -\sum_{{\bf t}\prec  \tilde{{\bf t}}} \lambda_{{\bf t}} z_{{\bf t}},  \]
 Using \lemref{lem: maxterm}, we have $\langle z_{\tilde{{\bf t}}}, \tilde{{\bf t}}(\hat{k}) \rangle =1$. We claim that  $\langle z_{{\bf t}}, \tilde{{\bf t}}(\hat{k}) \rangle =0$ for any decreasing sequence ${\bf t}\prec \tilde{{\bf t}}$ of $\mathcal{D}_{[k-1]}$. Once this is proved, we have 
 \[
  \lambda_{\tilde{{\bf t}}}=\langle  \lambda_{\tilde{{\bf t}}} z_{\tilde{{\bf t}}}, \tilde{{\bf t}}(\hat{k})  \rangle = -\sum_{{\bf t}\prec  \tilde{{\bf t}}} \lambda_{{\bf t}} \langle z_{{\bf t}}, \tilde{{\bf t}}(\hat{k}) \rangle=0,       
 \]
which contradicts our assumption. Therefore, $\{z_{{\bf t}} \mid {\bf t}\in \mathcal{D}_{[k-1]} \} $  is $\mathbb{Z}$-linearly independent.

It remains to prove our claim. Suppose that ${\bf t}=(t_1, \dots ,t_k)\prec \tilde{{\bf t}}=(\tilde{t}_1, \dots, \tilde{t}_k)$ in $\mathcal{D}_{[k-1]}$. Then by the definition of the lexicographical order, there exists an index $i$ such that $t_i\prec \tilde{t}_i$ and $t_j=\tilde{t}_j$ for $1\leq j\leq i-1$. If $i=k$, then by  definition \eqref{eq: Dk} of $\mathcal{D}_{[k-1]}$ we have two increasing chains $(t_k, t_{k+1}, \dots , t_n)$ and $(\tilde{t}_k, \tilde{t}_{k+1}, \dots , \tilde{t}_n)$ in the interval $[t_1t_2\cdots t_{k-1}, \gamma]=[\tilde{t}_1\tilde{t}_2\cdots \tilde{t}_{k-1}, \gamma]$.  This contradicts the EL-labelling property of $\mathcal{L}$. Therefore, we have $1\leq i\leq k-1$, and hence ${\bf t}(\hat{k}) \prec \tilde{{\bf t}}(\hat{k})$.  Now in view of \lemref{lem: maxterm},  $z_{{\bf t}}$ has a unique maximal term ${\bf t}(\hat{k})$. It follows that $\langle z_{{\bf t}},  \tilde{{\bf t}}(\hat{k})\rangle =0$ for any decreasing sequence ${\bf t}\prec \tilde{{\bf t}}$ of $\mathcal{D}_{[k-1]}$. 
\end{proof}

\begin{theorem}\label{thm: basisrksel}
 Let $\mathcal{B}= \bigoplus_{k=0}^n\mathcal{B}_k$ be as in \defref{def: AlgB}, and let $\mathcal{D}_{[k-1]}$ be as  in \eqref{eq: Dk}. 
    \begin{enumerate}
      \item  The elements $z_{{\bf t}}, {\bf t}\in \mathcal{D}_{[k-1]}$ form a basis for $\widetilde{H}_{k-2}(\mathcal{L}_{[k-1]})$ for $2\leq k\leq n$.  
      \item For $2\leq k\leq n$, the $\mathbb{Z}$-linear map
         \begin{equation*}
           d_k: \mathcal{B}_{k}\rightarrow \widetilde{H}_{k-2}(\mathcal{L}_{[k-1]}), \quad \beta_{{\bf t}} \mapsto z_{{\bf t}}, \quad \forall {\bf t}\in \bigcup_{w\in \mathcal{L}_{k}} {\rm Rex}_{T}(w)
        \end{equation*}
         is surjective.  
    \end{enumerate}
  \end{theorem}
\begin{proof}
 By  \propref{prop: dimtoprksel} and \propref{prop: selind}, the elements $z_{{\bf t}}, {\bf t}\in \mathcal{D}_{[k-1]}$ span a subgroup  of $\widetilde{H}_{k-2}(\mathcal{L}_{[k-1]})$ of maximal rank. To prove part (1), it suffices to show that every element of $\widetilde{H}_{k-2}(\mathcal{L}_{[k-1]})$ is a $\mathbb{Z}$-linear combination of those elements.  This can be proved by the method given in \thmref{thm: homint}, using the fact that $\langle z_{{\bf t}}, {\bf t}(\hat{k})  \rangle = 1$ and $\langle z_{{\bf t}}, \tilde{{\bf t}}(\hat{k}) \rangle=0$ for any decreasing sequence ${\bf t}\prec \tilde{{\bf t}}$ of $\mathcal{D}_{[k-1]}$. The former is given in \lemref{lem: bz} and the latter is the claim in the proof of \propref{prop: selind}. Part (2) follows from part (1). 
\end{proof}

\begin{example}\label{exam: A3con}
(Type $A_3$)
Let $\mathcal{L}$ be the NCP lattice given in \exref{exam: A3}, where  $W={\rm Sym}_4$ and $\gamma=(1243)$.
We determine the bases for the top homology groups $ \widetilde{H}_1(\mathcal{L}_{[2]})=\widetilde{H}_1(e, \gamma)$ and $\widetilde{H}_0(\mathcal{L}_{[1]})$ as follows. 

In view of the EL-labelling of $\mathcal{L}$  in \figref{fig: A3labelling}, we obtain the following sequences which satisfy the condition given in \eqref{eq: Dk}:
 \begin{align*}     
	\mathcal{D}_{[2]} =\{&  ((14), (34), (12)),\;  ((13), (14), (12)), \; ((24), (14), (34)), \\ 
	& ((13), (24), (14)),\;  ((23), (13), (24)) \},\\
	\mathcal{D}_{[1]}=\{ &((14), (12) ), \; ((13), (12)), \; ((34),(12) ), \; ((24), (34)), \; ((23), (24))\}.
\end{align*}
Note that $\mathcal{D}_{\gamma}=\mathcal{D}_{[2]}$. For each ${\bf t}=(t_1,t_2, t_3)\in \mathcal{D}_{\gamma}$, we have defined the following elements:
   \[
    \begin{aligned}
        \beta_{{\bf t}}&= (t_1, t_2, t_3)- (t_2, t_1^{t_2}, t_3)-(t_1, t_3, t_2^{t_3}) +(t_2, t_3, t_1^{t_2t_3})+ (t_3, t_1^{t_3}, t_2^{t_3}) -(t_3, t_2^{t_3}, t_1^{t_2t_3}),\\
        z_{{\bf t}} & = d(\beta_{{\bf t}})= (t_1, t_2)- (t_2, t_1^{t_2})-(t_1, t_3) +(t_2, t_3)+ (t_3, t_1^{t_3}) -(t_3, t_2^{t_3}).
    \end{aligned}
   \]
  We identify a sequence $(t_1, t_2)$ with the maximal chain $t_1<t_1t_2$ of $(e, \gamma)$. It is easily verified that $z_{{\bf t}}$ is a homology cycle of $\widetilde{H}_1(e, \gamma)$. By \thmref{thm: homint}, the elements $z_{{\bf t}}$ (resp. $\beta_{{\bf t}}$) for all ${\bf t}\in \mathcal{D}_{\gamma}$ form a basis for $\widetilde{H}_1(e, \gamma)$ (resp. $\mathcal{B}_{\gamma}$). 
  Similarly, the elements 
  \[ z_{(t_1,t_2)}=d(\beta_{(t_1,t_2)})= (t_1)-(t_2), \quad \forall (t_1,t_2)\in \mathcal{D}_{[1]}\]
  belong to  $\widetilde{H}_{0}(\mathcal{L}_{[1]})$, and  by \thmref{thm: basisrksel} they form a basis for $\widetilde{H}_{0}(\mathcal{L}_{[1]})$.
\end{example}

\section{Multiplicative structure on the Whitney homology}\label{sec: mult}

In this section we are motivated by  \cite[Section 3]{OS80}, where a combinatorial definition of the Orlik-Solomon algebra is introduced. It is shown in {\it loc. cit.} that as graded vector spaces, the Orlik-Solomon algebra is isomorphic to  the Whitney homology \cite{Bac75} of the intersection lattice of the corresponding hyperplane arrangement. This construction is not readily applicable to the NCP lattice, since the NCP lattice is not a geometric lattice. Nevertheless, we define a multiplicative structure on  $\mathcal{B}$ using a shuffle product, which makes extensive use of the Hurwitz action. This makes $\mathcal{B}$ into a finite dimensional graded associative algebra, which is isomorphic to the Whitney homology of the NCP lattice as a graded free abelian group. In addition, we prove that $\mathcal{B}=\bigoplus_{k=0}^n\mathcal{B}_k$  together with the linear maps $d_k$ form an acyclic chain complex.

\subsection{Multiplicative structure}
For each integer $k\geq 0$ we define free abelian groups  $\mathcal{T}_k$ as follows. Let $\mathcal{T}_0:=\mathbb{Z}$, and for $k\geq 1$ let $\mathcal{T}_k$  have a $\mathbb{Z}$-basis consisting of sequences $(t_1, t_2, \dots ,t_{k})$, where $t_i\in T$. Define $\mathcal{T}=\bigoplus_{k\geq 0} \mathcal{T}_k$. 
For each $k=1, \dots, n$ the braid group $B_k$ acts on $\mathcal{T}_k$ by the Hurwitz action as defined  in \eqref{eq: Braidact}. Note that $\mathcal{T}$ is isomorphic to the tensor space $T(U)=\bigoplus_{k\geq 0} U^{\otimes k}$, where $U$ is the abelian group freely spanned by all reflections of $W$.

We define an associative algebra structure on $\mathcal{T}$ as follows. Recall that a $(k,l)$-shuffle is a left minimal coset representative of ${\rm Sym}_{k}\times {\rm Sym}_l$ in ${\rm Sym}_{k+l}$. Denote by ${\rm Sh}(k,l)$ the set of all  $(k,l)$-shuffles of $1,2,\dots ,k+l$.  For any  two sequences of reflections ${\bf t}=(t_1,\dots,t_k)$ and ${\bf t}^{\prime}=(t_1^{\prime}, \dots, t_{l}^{\prime})$, we  define the star multiplication $*: \mathcal{T} \times \mathcal{T} \rightarrow \mathcal{T}$ by
\[ {\bf t}\ast {\bf t}^{\prime}:=\sum_{\pi \in {\rm Sh}(k,l)} {\rm sgn}(\pi) \underline{\pi}.(t_1,t_2, \dots ,t_k, t^{\prime}_1, t^{\prime}_2, \dots , t^{\prime}_l), \]
where $\underline{\pi}\in B_{k+l}$ is the lift of $\pi\in {\rm Sym}_{k+l}$   and the action is given by the Hurwitz action \eqref{eq: Braidact}. 

\begin{lemma}\label{lem: starass}
	The star multiplication $\ast$ is associative. 
\end{lemma}
\begin{proof}
It suffices to prove that $({\bf t}* {\bf t}')* {\bf t}''= {\bf t}* ({\bf t}'* {\bf t}'')$ for any three sequences of reflections ${\bf t}=(t_1, \dots, t_k), {\bf t}'=(t'_1, \dots, t'_l)$ and ${\bf t}''=(t''_1, \dots, t''_m)$.	
For convenience, set $\mathscr{B}_{k,l}:=\sum_{\pi \in {\rm Sh}(k,l)} {\rm sgn}(\pi) \underline{\pi}$. Then the associativity of the star multiplication follows from the following equation (cf. \cite[Proposition 9]{Ros90}):
\[
\mathscr{B}_{k+l,m} \mathscr{B}_{k,l} = \mathscr{B}_{k,l+m} \tau_k( \mathscr{B}_{l,m}),
\]
where $\tau_k: \mathbb{Z}B_{l+m} \rightarrow \mathbb{Z}B_{k+l+m}$ is the $\mathbb{Z}$-linear  homomorphism given by $\tau_k(\sigma_i)=\sigma_{k+i}$ for $1\leq i\leq l+m-1$. Expanding sums on both sides of the above equation,  the products $\underline{\pi}\underline{\eta}$ which appear  on each side are the lifts of the left minimal coset representatives $\pi\eta$ of ${\rm Sym}_{k}\times {\rm Sym}_l\times {\rm Sym}_m$ in ${\rm Sym}_{k+l+m}$, where $\pi\eta$ satisfies   $\ell(\pi \eta)=\ell(\pi)+ \ell(\eta)$ with $\ell$ being the usual length function  of ${\rm Sym}_{k+l+m}$. Since the lift to the braid group is independent of the reduced expression of the left minimal coset representative $\pi\eta$, we obtain the same sum on both sides of the equation. Therefore, the star multiplication $\ast$ in $\mathcal{T}$ is associative.
\end{proof}

For any sequence ${\bf t}=(t_1,\dots,t_k)$ of reflections, we define the following  element
\[
\xi_{{\bf t}}:=\sum_{\pi\in {\rm Sym}_k} {\rm sgn}(\pi)\, \underline{\pi}.(t_1,t_2,\dots,t_k)\in \mathcal{T}_k.
\]
By convention, we set $\xi_{{\bf t}}:=1$ if ${\bf t}$ is an empty sequence.
Clearly, if ${\bf t}\in {\rm Rex}_T(w)$ for some $w\in \mathcal{L}$, then we have $\xi_{{\bf t}}=\beta_{{\bf t}}$. To be more precise, let us define a $\mathbb{Z}$-linear map $p: \mathcal{T}\rightarrow \mathcal{T}$ by $p(1)=1$ and for any integer $k\geq 1$
\[
p((t_1,\dots,t_k))=\begin{cases}
(t_1,\dots,t_k), & \text{if $(t_1,\dots,t_k)\in \bigcup_{w\in \mathcal{L}_k} {\rm Rex}_T(w)$,}\\
0, & \text{otherwise.}	
\end{cases}
\]
Then we have $\beta_{{\bf t}}=p(\xi_{{\bf t}})$ if ${\bf t}\in  \bigcup_{w\in \mathcal{L}_k}{\rm Rex}_T(w)$ for $0\leq k\leq n$. 

\begin{lemma} \label{lem: star}
 Maintain the above notation. 
 \begin{enumerate}
 	\item  For any sequence ${\bf t}=(t_1,\dots,t_k)$ of reflections, we have 
 	$\xi_{{\bf t}}=(t_1)*(t_2)*\dots*(t_k)$.
 	\item For any $x,y\in \mathcal{T}$, we have $p(p(x)*p(y))=p(x*y)$.
 \end{enumerate}
\end{lemma}
\begin{proof}
  We use induction on $k$ to prove part (1). If $k=1$ then $\xi_{(t_1)}=(t_1)$ by definition. We identify ${\rm Sym}_{k-1}$ as a subgroup of ${\rm Sym}_k$ which permutes $\{1,2,\dots, k-1\}$. Note that ${\rm Sh}(k-1,1)$ is the set of left minimal coset representatives of ${\rm Sym}_{k-1}$ in ${\rm Sym}_k$.  Then 
  \[ 
   \begin{aligned}
      \xi_{{\bf t}}&=\sum_{\pi\in {\rm Sym}_k} {\rm sgn}(\pi) \underline{\pi}. {\bf t}= \sum_{\eta\in {\rm Sh}(k-1,1)} {\rm sgn}(\eta) \underline{\eta}. \bigg(\sum_{\pi^{\prime}\in {\rm Sym}_{k-1}} {\rm sgn}(\pi^{\prime}) \underline{\pi}^{\prime}. {\bf t} \bigg) \\
      &= \sum_{\eta\in {\rm Sh}(k-1,1)} {\rm sgn}(\eta)\, \underline{\eta}.((t_1)*(t_2)*\dots *(t_{k-1}),t_k) \\
      &= (t_1)*(t_2)*\dots*(t_k),
   \end{aligned}
 \]
 where we have used the induction hypothesis in the penultimate equation.
 
 For part (2), by linearity it suffices to prove the equation for ${\bf t}=(t_1, \dots, t_k)$ and ${\bf t}'=(t_1', \dots, t'_l)$. If $p({\bf t})=0$, then by the definition of $p$,  there exists no $w\in \mathcal{L}_k$ such that $w=t_1\cdots t_k$ is a  $T$-reduced expression. It follows that $t_1 \cdots t_k t_1' \cdots t_l'$ is not a $T$-reduced expression for any $w\in \mathcal{L}_{k+l}$. Hence $p((t_1, \dots, t_k, t_1', \dots, t_l'))=0$. Since  the product of all reflections in $\underline{\pi}.(t_1, \dots, t_k, t_1', \dots, t_l')$ remains the same for any $\pi\in {\rm Sh}(k,l)$, we have $p({\bf t}*{\bf t}')=0$. Therefore, the equation $p(p({\bf t})*p({\bf t}'))=p({\bf t}*{\bf t}')$ holds if $p({\bf t})=0$. This can be proved similarly if  $p({\bf t}')=0$. If $p({\bf t})$ and $p({\bf t}')$ are both nonzero, then by the definition of $p$ we have  $p({\bf t})={\bf t}$ and $p({\bf t}')={\bf t}'$, whence the equation holds.
\end{proof}

For each $k=0,1, \dots, n$, note that the map $p$ is a $\mathbb{Z}$-linear projection from $\mathcal{T}_k$ onto $C_{k-1}$. Recall that the latter is defined in \eqref{eq: chaingp}. Thus, we have $\mathcal{C}=p(\mathcal{T})=\bigoplus_{k=0}^n C_{k-1}$, and $\mathcal{C}$ is an abelian group freely spanned by  $(t_1, \dots ,t_k)\in \bigcup_{w\in \mathcal{L}_k} {\rm Rex}_T(w)$ for all $0\leq k\leq n$. For any two basis elements ${\bf t}$ and ${\bf t}'$ of $\mathcal{C}$, we define a product in $\mathcal{C}$ by 
\begin{equation}\label{eq: prod}
 {\bf t}{\bf t}':= p({\bf t}\ast {\bf t}'). 
\end{equation}
\begin{lemma}\label{lem: prodass}
	The multiplication in $\mathcal{C}$ defined by \eqref{eq: prod} is associative.
\end{lemma}
\begin{proof}
	It suffices to prove that $ ({\bf t}{\bf t}'){\bf t}''=  {\bf t}({\bf t}'{\bf t}'')$ for any ${\bf t}, {\bf t}', {\bf t}''\in \bigcup_{w\in \mathcal{L}} {\rm Rex}_T(w)$. On the left hand side we have
	\[ ({\bf t}{\bf t}'){\bf t}''=p(({\bf t}{\bf t}')*{\bf t}'' )=p(p({\bf t}\ast {\bf t}' )* p({\bf t}''))= p(({\bf t}\ast {\bf t}') * {\bf t}'') ,\]
	where in the second equation we have used $p({\bf t}'')={\bf t}''$ and the last equation follows from \lemref{lem: star}.
	Similarly, we have $ {\bf t}({\bf t}'{\bf t}'')=p({\bf t}\ast ({\bf t}' * {\bf t}''))$. Then the associativity of the product in $\mathcal{C}$ follows from that of the star multiplication, which is given in \lemref{lem: starass}. 
\end{proof}

Recall from \defref{def: AlgB} that $\mathcal{B}=\bigoplus_{w\in \mathcal{L}}\mathcal{B}_w$ is a free abelian subgroup of $\mathcal{C}$. For any $w\in \mathcal{L}_k$ ($0\leq k\leq n$), the subgroup $\mathcal{B}_w\subseteq C_{k-1}$ is spanned by all elements $\beta_{\bf t}, {\bf t}\in {\rm Rex}_T(w)$. The following lemma states that $\mathcal{B}$ is closed under the multiplication \eqref{eq: prod}.

\begin{lemma}\label{lem: prod}
 Let $u,w\in \mathcal{L}$. Then for any ${\bf t}\in {\rm Rex}_T(u)$ and  ${\bf t}'\in {\rm Rex}_T(w)$, we have 
 \[
    \beta_{{\bf t}} \beta_{{\bf t}'}=
\begin{cases}
	\beta_{({\bf t},{\bf t}')}, & \text{if } uw\leq \gamma \text{\, and \,}\ell_{T}(uw)=\ell_{T}(u)+ \ell_{T}(w),\\
	0,  &\text{otherwise.}
\end{cases}
 \]
 This product  is associative, i.e. $(\beta_{{\bf t}} \beta_{{\bf t}'})\beta_{{\bf t}''}= \beta_{{\bf t}} (\beta_{{\bf t}'}\beta_{{\bf t}''})$ for any ${\bf t}, {\bf t}', {\bf t}''\in \bigcup_{w\in \mathcal{L}} {\rm Rex}_T(w)$.
\end{lemma}
\begin{proof}
 Assume that ${\bf t}=(t_1,\dots, t_k)\in {\rm Rex}_T(u)$ and ${\bf t}'=(t_1',\dots, t_l')\in {\rm Rex}_T(w)$. Using definition \eqref{eq: prod} and \lemref{lem: star} we have 
 \[
 \begin{aligned}
     \beta_{{\bf t}} \beta_{{\bf t}'}&= p(   \beta_{{\bf t}}\ast \beta_{{\bf t}'}) =  p(   p(\xi_{{\bf t}})\ast p(\xi_{{\bf t}'}))= p( \xi_{{\bf t}}\ast \xi_{{\bf t}'} )\\
     &= p( (t_1)*\dots*(t_k)* (t'_1)*\dots*(t'_k)  )\\
     &=p(\xi_{({\bf t},{\bf t}')}).
 \end{aligned}
 \]
 Now if $ uw\leq \gamma$ and $\ell_{T}(uw)=\ell_{T}(u)+ \ell_{T}(w)$, then $({\bf t},{\bf t}') \in {\rm Rex}_T(uw)$ and hence $\beta_{{\bf t}} \beta_{{\bf t}'}=p(\xi_{({\bf t},{\bf t}')})=\beta_{({\bf t},{\bf t}')}$; otherwise, $uw$ is not an element of $\mathcal{L}$ or $({\bf t},{\bf t}')$ is not $T$-reduced, whence by the definition of $p$ we have  $\beta_{{\bf t}} \beta_{{\bf t}'}=p(\xi_{({\bf t},{\bf t}')})=0$. Since $\beta_{{\bf t}}\in \mathcal{C}$ for any ${\bf t}\in \bigcup_{w\in \mathcal{L}} {\rm Rex}_T(w)$, the associativity of the product in $\mathcal{B}$ follows from \lemref{lem: prodass}. This completes the proof.
\end{proof}

\begin{proposition}\label{prop: Balg}
  The free abelian group $\mathcal{B}=\bigoplus_{w\in \mathcal{L}}\mathcal{B}_w$  is a finite dimensional graded associative algebra over $\mathbb{Z}$,  generated by homogeneous elements $\beta_{t}, t\in T$ of degree $1$ with multiplication given as in \lemref{lem: prod}.
\end{proposition}
\begin{proof}
 For any ${\bf t}=(t_1,t_2,\dots, t_k)\in {\rm Rex}_T(w)$ with $w\in \mathcal{L}_k$, using \lemref{lem: prod} we have 
 $\beta_{{\bf t}}= \beta_{t_1}\beta_{t_2}\cdots \beta_{t_k}$. As the algebra $\mathcal{B}$ is $\mathbb{Z}$-linearly spanned by all elements $\beta_{{\bf t}}$, it is a finite dimensional algebra generated by  $\beta_{t}, t\in T$. 
\end{proof}

\begin{remark}\label{rmk: Whom}
  Baclawski introduced the Whitney homology $WH_*(P)$ for any poset $P$ with a least element $\hat{0}$ \cite{Bac75}. Subsequently,  Bj\"{o}rner proved that the  Whitney homology can be understood in terms of the usual homology of intervals \cite[Theorem 5.1]{Bjo82}, that is,
   $WH_k(P)\cong \bigoplus_{x\in P-\{\hat{0}\}}\widetilde{H}_{k-2}(\hat{0},x)$.
 We  define $WH_0(P):=\mathbb{Z}$. In our case $\mathcal{L}$ is Cohen-Macaulay, so the right side of the above isomorphism can be simplified as follows:
   \begin{equation*}
   	   WH_k(\mathcal{L})\cong \bigoplus_{w\in \mathcal{L}_k}\widetilde{H}_{k-2}(e,w),\quad 1\leq k\leq n.
   \end{equation*}
   Then it follows from the isomorphism \eqref{eq: Bwiso} that $WH_{k}(\mathcal{L})\cong \mathcal{B}_k, 0\leq k\leq n$ as free abelian groups. Therefore,  \lemref{lem: prod} defines a multiplicative structure on the Whitney homology $  WH(\mathcal{L}):=\bigoplus_{0\leq k\leq n}   WH_k(\mathcal{L})$.
\end{remark}

 It is worthwhile to find some relations among the generators $\beta_t$ and obtain a presentation of $\mathcal{B}$. We illustrate this with an example. 

\begin{example}\label{exam: quadrel}
  Suppose that ${\bf t}=(t_1,t_2)\in {\rm Rex}_{T}(w)$ for some $w\in \mathcal{L}_2$. Then 
  $\beta_{{\bf t}}=\beta_{t_1}\beta_{t_2}=(t_1)*(t_2)=(t_1,t_2)-(t_2, t_2t_1t_2). $ One can verify directly that 
   \[
\begin{aligned}
   \beta_{t_1}\beta_{t_2}+\beta_{t_2}\beta_{t_1}&=0, \ \text{if } t_1t_2=t_2t_1,\\
  \beta_{t_1}\beta_{t_2}+ \beta_{t_2}\beta_{t_2t_1t_2}+\beta_{t_2t_1t_2}\beta_{t_1}&=0, \ \text{if } t_1t_2t_1=t_2t_1t_2.\\
\end{aligned}
  \]
  \end{example}

\begin{proposition}\label{prop: quadrel}
   We have the following quadratic relations among generators $\beta_{t}, t\in T$:
  \begin{enumerate}
  \item $\beta_t^2= \beta_{t_1}\beta_{t_2}=0$   for all $t\in T$ and $t_1,t_2\in T$ with  $t_1t_2 \not \leq  \gamma$.
  \item  For any  $w\in \mathcal{L}_2$, we have 
 \begin{equation*}\label{eq: quadrel2}
   \sum_{(t_1,t_2)\in {\rm Rex}_{T}(w)} \beta_{t_1}\beta_{t_2}=0.
 \end{equation*}
  \end{enumerate}
\end{proposition}
\begin{proof}
 Part (1) follows directly from  \lemref{lem: prod}. For part (2), note that each interval $[e,w]$ of $\mathcal{L}$ with $\ell_{T}(w)=2$ is isomorphic to the  noncrossing partition lattice associated to a dihedral group (see \cite[Lemma 1.4.3]{Bes03} and \cite[Proposition 2.6.11]{Arm09}). Without loss of generality,  we may assume $w=s_1s_2$ for two simple reflections $s_1$ and $s_2$. If the order of $s_2s_1$ is $m$,  then the reflections preceding $w$ are  $r_i=s_1(s_2s_1)^{i-1}, 1\leq i \leq m$. Then we have $${\rm Rex}_{T}(w)=\{(r_1,r_m), (r_2, r_1), (r_3, r_2),\dots, (r_m, r_{m-1}) \}.$$  It follows that $r_1r_2r_1=r_m, r_mr_1r_m=r_{m-1}$ and $r_kr_{k+1}r_k=r_{k-1},  2\leq k\leq m-1$. Therefore, we obtain 
\[
\begin{aligned}
\sum_{(t_1,t_2)\in {\rm Rex}_{T}(w)} \beta_{t_1}\beta_{t_2}=&\, \beta_{ r_1} \beta_{r_m}+ \beta_{ r_2} \beta_{r_1}+\dots + \beta_{ r_m} \beta_{r_{m-1}}\\
=&\, (r_1, r_m)- (r_m, r_m r_1 r_m) + (r_2, r_1)- (r_1, r_1r_2r_1)+ \dots + \\
 &\,  (r_m, r_{m-1})- (r_{m-1}, r_{m-1}r_{m} r_{m-1})\\
=&\, 0.
\end{aligned}
\]
This completes the proof.
\end{proof}

\begin{remark}\label{rmk: quadrel}
  In fact,  all relations among $\beta_{t}, t\in T$ are generated by the quadratic relations given in \propref{prop: quadrel}. We illustrate this with the following example. For a general proof, see \cite{LZ22} or \cite{Zha20}.
\end{remark}

\begin{example}\label{exam: dihbasis}
(Dihedral group)	Let $W$ be the  dihedral group as defined in \exref{exam: dih}. By \propref{prop: Balg}, the algebra $\mathcal{B}$ is generated by $\beta_{t}$ for all $t\in T$, where $T$ is totally ordered as in \eqref{eq: totord}. In view of \propref{prop: quadrel}, the generators $\beta_t$ satisfy the following relations: 
	 \[
	 \begin{aligned}
	\beta_{t_1}\beta_{t_i}&=0, \quad 1\leq i\leq m-1,\\
	 	\beta_{t_i}\beta_{t_j}&=0, \quad 2\leq i \leq m, 1\leq j \neq i-1 \leq m, \\
	 	\beta_{t_1} \beta_{t_m}&+\beta_{t_2} \beta_{t_1}+\dots +\beta_{t_m} \beta_{t_{m-1}}=0. 
	 \end{aligned}
	 \]
	
  Recall that $\mathcal{B}=\mathcal{B}_0 \oplus \mathcal{B}_{1}\oplus \mathcal{B}_2$, where $\mathcal{B}_0=\mathbb{Z}$ and  $\mathcal{B}_1$ has a $\mathbb{Z}$-linear basis consisting of all $\beta_t, t\in T$.  By the  relations above,  $\mathcal{B}_2=\mathcal{B}_{\gamma}$ is  spanned by $m-1$ elements $\beta_{(t_i,t_{i-1})}=\beta_{t_i}\beta_{t_{i-1}}, 2\leq i \leq m$, where the sequences  $(t_i,t_{i-1})$ are all decreasing under the total order of $T$. These elements form a basis for $\mathcal{B}_2$ by \thmref{thm: homint}.
  Therefore, the relations above are all quadratic relations among $\beta_{t}, t\in T$.
\end{example}

\subsection{An acyclic chain complex}

Recall that  $\mathcal{B}=\bigoplus_{k=0}^{n}\mathcal{B}_k$ is a $\mathbb{Z}$-graded algebra. Noting that  $z_{{\bf t}}=\sum_{i=1}^{k}(-1)^{k-i}\beta_{{\bf t}(\hat{i})}\in \mathcal{B}_{k-1}$ by \eqref{eq: zexp}, we have a map
\begin{equation}\label{eq: dB}
   d_k: \mathcal{B}_{k}\rightarrow \mathcal{B}_{k-1}, \quad \beta_{{\bf t}} \mapsto z_{{\bf t}}, \quad \forall {\bf t}\in \bigcup_{w\in \mathcal{L}_{k}} {\rm Rex}_{T}(w),\quad  0\leq k\leq n.
\end{equation}
In particular, $z_{t}=d_1(\beta_t)=1$ for any $t\in T$.

\begin{lemma}\label{lem: dprop}
Let $d$ be as in \eqnref{eq: dB}. Then 
   \begin{enumerate}
    \item  We have $d^2=0$, whence we have the following chain complex:
        \begin{equation*}\label{eq: chaincomp}
  \xymatrix{
   0  \ar[r] &  \mathcal{B}_{n} \ar[r]^-{d_{n}} &  \mathcal{B}_{n-1}\ar[r]^-{d_{n-1}} &  \cdots  \ar[r]^-{d_1} &  \mathcal{B}_{0}\ar[r] & 0.
}
\end{equation*}

    \item  Let $w\in \mathcal{L}_{k}$ with $2\leq k\leq n$. Then for each $i=1,2,\dots, k-1$, we have
            \[ d(\beta_{\mathbf{t}})=(-1)^{k-i}(d\beta_{(t_1,\dots,t_i)})\beta_{(t_{i+1},\dots,t_k)}+\beta_{(t_1\dots, t_i)} (d\beta_{(t_{i+1},\dots,t_k)}) \]
    for any ${\bf t}=(t_1,t_2,\dots, t_k)\in {\rm Rex}_{T}(w)$. 
   \end{enumerate}
\end{lemma}
\begin{proof}
  By \eqref{eq: zexp} for any ${\bf t}=(t_1,t_2,\dots, t_k)\in {\rm Rex}_{T}(w)$ with $w\in \mathcal{L}$, we have 
  \begin{equation}\label{eq: dprop}
    d(\beta_{{\bf t}})=z_{{\bf t}}=\sum_{i=1}^k (-1)^{k-i} \beta_{{\bf t}(\hat{i})}= -d(\beta_{(t_1,\dots,t_{k-1})})\beta_{t_k}+\beta_{(t_1,\dots,t_{k-1})}.
  \end{equation}
  For part (1), we use induction on the length of ${\bf t}$. Clearly, $d^2(1)=0$. Applying $d$ to both sides of  \eqref{eq: dprop}, we have 
  \begin{equation}\label{eq: applyd}
  	  d^2(\beta_{{\bf t}})= -d(d(\beta_{(t_1,\dots,t_{k-1})})\beta_{t_k}) + d(\beta_{(t_1,\dots,t_{k-1})}).
  \end{equation}
Noting that  $d(\beta_{(t_1,\dots,t_{k-1})})= \sum_{i=1}^{k-1}(-1)^{k-1-i} \beta_{(t_1,\dots, \hat{t}_i,\dots,t_{k-1})}$, we have
  \[
  \begin{aligned}
  	d(d(\beta_{(t_1,\dots,t_{k-1})})\beta_{t_k}) &=\sum_{i=1}^{k-1}(-1)^{k-1-i}d ( \beta_{(t_1,\dots, \hat{t}_i,\dots,t_{k-1})} \beta_{t_k}) \\
  	 &= \sum_{i=1}^{k-1}(-1)^{k-1-i} (- d ( \beta_{(t_1,\dots, \hat{t}_i,\dots,t_{k-1})} )  \beta_{t_k}+  \beta_{(t_1,\dots, \hat{t}_i,\dots,t_{k-1})} ) \\
  	 &= -d(d(\beta_{(t_1,\dots,t_{k-1})}))\beta_{t_k}+ d(\beta_{(t_1,\dots,t_{k-1})})\\
  	 &= d(\beta_{(t_1,\dots,t_{k-1})}),
  \end{aligned}
  \]
 where the second equation is a consequence of \eqref{eq: dprop} and the last equation follows from the induction hypothesis. Using the above equation in \eqref{eq: applyd}, we obtain  $ d^2(\beta_{{\bf t}})=0$. 

  For part (2) we  also use induction on $k$. This is trivial if $k=2$. For general $k>2$, note that if $i=k-1$ we obtain \eqref{eq: dprop}. For $i\leq k-2$, by induction we have 
  \[ d(\beta_{(t_1,\dots, t_{k-1})})=(-1)^{k-1-i}(d\beta_{(t_1,\dots,t_i)})\beta_{(t_{i+1},\dots,t_{k-1})}+\beta_{(t_1,\dots, t_i)} (d\beta_{(t_{i+1},\dots,t_{k-1})}). \]
  Using this in \eqref{eq: dprop} we have 
   \begin{align*}
     d(\beta_{{\bf t}})=&(-1)^{k-i} d(\beta_{(t_1,\dots,t_i)})\beta_{(t_{i+1},\dots,t_k)}+\beta_{(t_1,\dots,t_i)}(-d(\beta_{(t_{i+1},\dots, t_{k-1})})\beta_{t_k}+\\
       &\beta_{(t_{i+1},\dots, t_{k-1})})\\
    =& (-1)^{k-i}(d\beta_{(t_1,\dots,t_i)})\beta_{(t_{i+1},\dots,t_k)}+\beta_{(t_1,\dots, t_i)} (d\beta_{(t_{i+1},\dots,t_k)})
   \end{align*}
for any $i\leq k-2$. Therefore, the equation holds for $1\leq i\leq k-1$. 
\end{proof}

\begin{remark}
Note that $(\mathcal{B},d)$ is not a differential graded algebra, as  $d$ does not preserve the multiplication as given in \lemref{lem: prod}. For instance, let $t_1,t_2\in T$ such that $t_1t_2\not \leq \gamma$, then $\beta_{t_1}\beta_{t_2}=0$ by \propref{prop: quadrel} and hence $d(\beta_{t_1}\beta_{t_2})=0$, while $-d(\beta_{t_1})\beta_{t_2}+\beta_{t_1}d(\beta_{t_2})=-\beta_{t_2}+\beta_{t_1}\neq 0$. 
Part (2) of \lemref{lem: dprop} is only true for these nonzero elements 
$\beta_{{\bf t}}=\beta_{t_1}\cdots \beta_{t_k}$  with $(t_1, \dots, t_k )\in {\rm Rex}_{T}(w)$ rather than an arbitrary product of  $\beta_t$'s. 
\end{remark}

\begin{proposition}\label{prop: ayccomp}
  Let  $d_k$ be the linear map as in \eqref{eq: dB}  for $1\leq k\leq n$.
  \begin{enumerate}
     \item We have ${\rm Im}\,d_k=\widetilde{H}_{k-2}(\mathcal{L}_{[k-1]})$ for $1\leq k\leq n$, and ${\rm Ker}\, d_k = \widetilde{H}_{k-1}(\mathcal{L}_{[k]})$ for $1\leq k\leq n-1$ and ${\rm Ker}\, d_n=0$. It follows that
       \[
\mathcal{B}_k \cong \begin{cases}
  \widetilde{H}_{k-2}(\mathcal{L}_{[k-1]})\oplus  \widetilde{H}_{k-1}(\mathcal{L}_{[k]}), & \quad 1\leq k\leq n-1,\\
   \widetilde{H}_{n-2}(\mathcal{L}_{[n-1]}), & \quad k=n.
\end{cases}
       \]
     \item   The chain complex $(\mathcal{B}_k,d_k)$ is acyclic.
  \end{enumerate}
\end{proposition}

\begin{proof}
 We just need to prove part (1), which leads to part (2) immediately.  Noting that  $\widetilde{H}_{n-2}(\mathcal{L}_{[n-1]})=\widetilde{H}_{n-2}(e,\gamma)$ is a free abelian subgroup of $\mathcal{B}_{n-1}$ and $\mathcal{B}_n=\mathcal{B}_{\gamma}$, the linear map $d_n: \mathcal{B}_n\rightarrow  \widetilde{H}_{n-2}(\mathcal{L}_{[n-1]})$ is an isomorphism by part (3) of \thmref{thm: homint}. Therefore, we have $ {\rm Im}\,d_n=\widetilde{H}_{n-2}(\mathcal{L}_{[n-1]})$ and ${\rm Ker}\, d_n =0$. 

 If $2\leq k\leq n-1$, then by part (2) of \thmref{thm: basisrksel}, the linear map $d_k: \mathcal{B}_k\rightarrow \widetilde{H}_{k-2}(\mathcal{L}_{[k-1]})$ is surjective. It follows that ${\rm Im}\, d_k= \widetilde{H}_{k-2}(\mathcal{L}_{[k-1]})$ for $2\leq k\leq n-1$,  and ${\rm Im}\, d_1=\mathbb{Z}$ as $d(\beta_t)=1$ for any $t\in T$. It remains to prove that ${\rm Ker}\, d_k = \widetilde{H}_{k-1}(\mathcal{L}_{[k]})$.  

  By part (1) of \lemref{lem: dprop} we have $d^2=0$, whence ${\rm Im}\, d_{k+1}=\widetilde{H}_{k-1}(\mathcal{L}_{[k]})\subseteq {\rm Ker}\, d_k$ for $1\leq k \leq n-1$. We now prove that $\widetilde{H}_{k-1}(\mathcal{L}_{[k]})$ and ${\rm Ker}\, d_k$ have equal rank. 
   By the definition of the M\"obius function  and \corref{coro: rankBw}, we have
\[
\mu(\hat{\mathcal{L}}_{[k]})= -\sum_{i=0}^{k}\sum_{w\in \mathcal{L}_i}\mu(w)= \sum_{i=0}^{k}(-1)^{i+1} {\rm rank}\, \mathcal{B}_i. 
\]
It follows from \propref{prop: dimtoprksel} that for each $k$
$${\rm rank}\, \widetilde{H}_{k-1}(\mathcal{L}_{[k]})=(-1)^{k+1}\mu(\hat{\mathcal{L}}_{[k]})=  \sum_{i=0}^{k}(-1)^{k-i} {\rm rank}\, \mathcal{B}_i. $$
Now using the equation ${\rm rank}\, \widetilde{H}_{k-2}(\mathcal{L}_{[k-1]})=  \sum_{i=0}^{k-1}(-1)^{k-1-i} {\rm rank}\, \mathcal{B}_i $ in the right side of the above equation, we have 
\[
{\rm rank}\, \widetilde{H}_{k-1}(\mathcal{L}_{[k]})= {\rm rank}\, \mathcal{B}_k- {\rm rank}\, \widetilde{H}_{k-2}(\mathcal{L}_{[k-1]})= {\rm rank}\, \mathcal{B}_k- {\rm rank}\, {\rm Im}\, d_k={\rm rank} \, {\rm Ker}\, d_k.
\]
 Thus, $\widetilde{H}_{k-1}(\mathcal{L}_{[k]})$ is a subgroup of ${\rm Ker}\,{d_k}$ of maximal rank. It remains to show that every element of ${\rm Ker}\, d_k$ is a $\mathbb{Z}$-linear combination of the basis elements $z_{{\bf t}}, {\bf t}\in \mathcal{D}_{[k]}$ of $\widetilde{H}_{k-1}(\mathcal{L}_{[k]})$. We omit this since the proof is analogous to that of \thmref{thm: basisrksel}. Therefore, we have ${\rm Ker}\, d_k = \widetilde{H}_{k-1}(\mathcal{L}_{[k]})$ for $1\leq k\leq n-1$. 
\end{proof}

\begin{remark}
   Recall from \remref{rmk: Whom} that $\mathcal{B}_k$ is isomorphic to the $k$-th Whitney homology $WH_k(\mathcal{L})$ as free abelian groups.  Part (1) of \propref{prop: ayccomp} implies that the $k$-th Whitney homology $WH_k(\mathcal{L})$ is isomorphic to $\widetilde{H}_{k-2}(\mathcal{L}_{[k-1]})\oplus  \widetilde{H}_{k-1}(\mathcal{L}_{[k]})$.  Sundaram proved this fact  in a more general context where the lattice $\mathcal{L}$ is replaced with a Cohen-Macaulay poset \cite[Proposition 1.9]{Sun94}. 
\end{remark}

We now summarise the main properties of $\mathcal{B}$. Combining Theorems  \ref{thm: homint}, \ref{thm: basisrksel}, Propositions \ref{prop: Balg}, \ref{prop: ayccomp}, and \corref{coro: rankBw}, we have the following theorem.

\begin{theorem}\label{thm: main}
   Let $\mathcal{B}=\bigoplus_{w\in \mathcal{L}}\mathcal{B}_{w}$  be the  abelian group spanned by $\beta_{{\bf t}}$ for all sequences ${\bf t}\in \bigcup_{w\in \mathcal{L}} {\rm Rex}_{T}(w)$, and let $\mathcal{B}_k=\bigoplus_{w\in \mathcal{L}_k} \mathcal{B}_w$. Then we have
\begin{enumerate}
\item The free abelian group $\mathcal{B}=\bigoplus_{k=0}^n\mathcal{B}_k$ is a  finite dimensional graded  algebra over $\mathbb{Z}$ generated by the homogeneous elements $\beta_{t}, t\in T$ of degree $1$, with multiplication defined by 
 \[
    \beta_{{\bf t}} \beta_{{\bf t}'}=
\begin{cases}
  \beta_{({\bf t},{\bf t}')}, & \text{if } uw\leq \gamma \text{\, and \,}\ell_{T}(uw)=\ell_{T}(u)+ \ell_{T}(w),\\
  0,  &\text{otherwise.}
\end{cases}
 \]
 for any ${\bf t}\in {\rm Rex}_T(u)$ and  ${\bf t}'\in {\rm Rex}_T(w)$ with $u,w\in \mathcal{L}$.

\item  For $1\leq k\leq n$,  the $\mathbb{Z}$-linear map 
   $ d_k: \mathcal{B}_{k}\rightarrow \widetilde{H}_{k-2}(\mathcal{L}_{[k-1]})$ defined by 
   \begin{equation*}
     d_{k}(\beta_{(t_1, \dots, t_k)})=\sum_{i=0}^k(-1)^{k-i}\beta_{(t_1, \dots, \hat{t}_i, \dots, t_k)}, \quad \forall \beta_{(t_1, \dots, t_k)}\in \mathcal{B}_k
   \end{equation*}
 is surjective. Moreover, we have the isomorphisms of free abelian groups:
   \[
 \mathcal{B}_k \cong \begin{cases}
 	\widetilde{H}_{k-2}(\mathcal{L}_{[k-1]})\oplus  \widetilde{H}_{k-1}(\mathcal{L}_{[k]}), & \quad 1\leq k\leq n-1,\\
 	\widetilde{H}_{n-2}(\mathcal{L}_{[n-1]}), & \quad k=n.
 \end{cases}
 \]

 Hence the chain complex $(\mathcal{B}_k, d_k)$ is acyclic. 

\item For each $w\in \mathcal{L}$ with $k=\ell_{T}(w)\geq 1$, we have $\mathcal{B}_w \cong \widetilde{H}_{k-2}(e,w)$ as free abelian groups, where the isomorphism is given by the restriction of $d_k$ to $\mathcal{B}_w$.

\item For each $w\in \mathcal{L}$, the elements $\beta_{{\bf t}}, {\bf t}\in \mathcal{D}_w$ constitute a $\mathbb{Z}$-basis for $ \mathcal{B}_w$.

\item The algebra $\mathcal{B}$ has the Poincar\'e polynomial 
   \[ {\rm Poin}_{\mathcal{B}}(q):=\sum_{k=0}^n {\rm rank}\, \mathcal{B}_k\, q^k= \sum_{w\in \mathcal{L}} \mu(w) (-q)^{\ell_{T}(w)}, \]
   where $\mu$ is the M\"{o}bius function of $\mathcal{L}$. 
\end{enumerate}  
\end{theorem}

\section{Application to the Milnor fibre}\label{sec: app1}
 In this section, we will give two chain complexes which compute the integral homology of the Milnor fibre of the reflection arrangement and the Milnor fibre of the discriminant of $W$, respectively. Computational results will be tabulated in Appendix \ref{sec: appd}.

\subsection{The NCP model of the Milnor fibre}
We refer to \cite{BFW18,DL16,Mil68,Zha20} for background concerning the Milnor fibre.
Recall that $W$ is a finite real reflection group whose action on $\mathbb{R}^{n}$ can be complexified to an action on $\mathbb{C}^{n}=\mathbb{R}^{n}\otimes \mathbb{C}$. For each $t\in T$, let $H_t \subset  \mathbb{C}^{n}$ be the complexified hyperplane and $\lambda_t: \mathbb{C}^{n}\rightarrow \mathbb{C}$  a linear form such that $H_{t}={\rm Ker}(\lambda_t)$. The finite collection of reflecting hyperplanes 
\[ \mathcal{A}:=\{H_{t}\mid t\in T \} \]
is called the \emph{reflection arrangement} of $W$ in $\mathbb{C}^{n}$, and the space $M= \mathbb{C}^n - \bigcup_{t\in T} H_t$ is  called the \emph{hyperplane complement} of $W$.  We define the  homogeneous polynomials $Q$ and $Q_0$ by
\begin{equation}\label{eq: polyQ}
  Q_0:= \prod_{t\in T} \lambda_t, \quad Q:=Q_0^2=\prod_{t\in T} \lambda_t^2. 
\end{equation}
Note that for any simple reflection $s\in S$, we have $s\lambda_s=-\lambda_s$ and $s$ permutes all other linear forms $\lambda_t, s\neq t\in T$. It follows that $ wQ_0={\rm det}(w) Q_0$ for any $ w\in W$,  and hence $Q=Q_0^2$ is a $W$-invariant polynomial. 

The polynomial $Q_0$ has non-isolated singular points. The Milnor fibration of $Q_0$ is the locally trivial fibration $Q_0: M\rightarrow \mathbb{C}-\{0\}$, with the Milnor fibre $F_{Q_0}= Q_{0}^{-1}(1)$ and the geometric monodromy $\zeta: F_{Q_0}\rightarrow F_{Q_0}, h(x_1, \dots ,x_n)= (\zeta_dx_1, \dots , \zeta_dx_n)$, where $\zeta_d= {\rm exp}(2\pi i/d)$ with $d=|T|$. The  geometric monodromy $\zeta$ realises $F_{Q_0}$ as a $d$-fold cyclic covering map of the projective space $\mathbb{P}(M)$ of $M$, i.e. $F_{Q_0}/\mu_d\cong \mathbb{P}(M)$, where $\mu_d$ is the monodromy group of order $d$ generated by $\zeta$. 

The Milnor fibre $F_Q$ is a disjoint union of two connected components $F_{Q_0}$ and $Q_{0}^{-1}(-1)$, which are homeomorphic to each other \cite{DL16}. The Coxeter group $W$ acts freely on $F_Q$.

\begin{definition}\label{def: ncpmodfull}
  \cite{BFW18}The NCP model $\mathcal{F}$ of the Milnor fibre $F_{Q}=Q^{-1}(1)$ is the $(n-1)$-dimensional  finite simplicial complex whose  $k$-simplices are of the form
    \[ (m,w, e< w_1< \dots < w_k), \quad 0\leq m < n-\ell_{T}(w_k), w\in W, \]
    where  $e< w_1< \dots < w_k$ is a chain of $\mathcal{L}$. Each $k$-simplex has $k+1$ maximal faces 
    $( m,w, e< w_1< \dots < \widehat{w_i}<\dots < w_k )$ for $1\leq i\leq k$  
    and the remaining one 
  $
    (m+\ell_{T}(w_1), ww_1, e< w_1^{-1}w_2 < \dots < w_1^{-1}w_k)$.
\end{definition}

Note that by the condition on the integer $m$, the  element $w_k$ in the $k$-simplex $(m,w, e< w_1< \dots < w_k)$ should be strictly less than the Coxeter element  $\gamma$ of $\mathcal{L}$. The Coxeter group $W$ acts on the $k$-simplex by left multiplication on $w$. The generator of the monodromy group  of $F_{Q}$ sends the $k$-simplex to  $(m-1,w, e< w_1< \dots < w_k)$ if $0<m<n-\ell_{T}(w_k)$ and $(-1)^k (\ell_{T}(w_1)-1, ww_1, e< w_1^{-1}w_2< \dots <w_1^{-1}w_k< w_1^{-1}\gamma)$ if $m=0$. One can check that this monodromy group has order $nh$, with $h$ being the Coxeter number, i.e. the order of the Coxeter element.

We define the $k$-th boundary map of $\mathcal{F}$ by 
 \begin{equation}\label{eq: boumapQ}
   \begin{aligned}
\partial_k(m, w, e< w_1 < \dots < w_k)= &(m+ \ell_T(w_1), ww_1, e < w_1^{-1}w_2< \dots < w_1^{-1}w_k) \\
 & +\sum_{i=1}^k (-1)^i ( m,w, e< w_1< \dots < \widehat{w_i}<\dots < w_k ).
\end{aligned}
 \end{equation}
 Then the $k$-th integral homology of $\mathcal{F}$ is  defined by
\[
   H_k(\mathcal{F}; \mathbb{Z}):= {\rm Ker}\, \partial_k/ {\rm Im}\, \partial_{k+1}, \quad 0\leq k\leq n-1.
\]
It is easily verified that the actions of both $W$ and $\mu_d$ commute with $\partial_k$, and hence each $ H_k(\mathcal{F}; \mathbb{Z})$ is a $W\times \mu_d$-module. 

\subsection{A filtration of the NCP model}

We may define a filtration of $\mathcal{F}$ as follows. Define $\mathcal{F}_i$  to be the union of the  $i$-dimensional simplices of the  form
$
(n-i-1, w, e< w_1 <\dots < w_i)  
$
together with their faces, where $w\in W$ and  $w_1 < \dots < w_i$ is a maximal chain of the subposet $\mathcal{L}_{[i]}$ with $\ell_{T}(w_j)=j$ for $1\leq j\leq i$. Note that $\mathcal{F}_i$ is an $i$-dimensional subcomplex of $\mathcal{F}$.  We obtain the  filtration of $\mathcal{F}$ by the subcomplexes $\mathcal{F}_i$
\begin{equation}\label{eq: Ffiltration}
   \mathcal{F}=\mathcal{F}_{n-1}\supset \mathcal{F}_{n-2}\supset \dots \supset \mathcal{F}_1 \supset \mathcal{F}_0\supset \mathcal{F}_{-1}=\emptyset.
\end{equation}

We use the spectral sequence of this filtration to calculate the homology of $\mathcal{F}$; refer to \cite[Chapter III]{BT82} for the spectral sequence. For each $k\leq i$, let $C_k(\mathcal{F}_i)$ be the $k$-th chain group of the simplicial subcomplex $\mathcal{F}_i$, that is, the free abelian group spanned by $k$-simplices of $\mathcal{F}_i$. In the spectral sequence associated to the filtration, the $E^0$-page  is given by the following free $\mathbb{Z}$-modules:
\[E^{0}_{i,j}= C_{i+j}(\mathcal{F}_i)/ C_{i+j}(\mathcal{F}_{i-1})\cong C_{i+j, n-i-1},\]
where  $C_{i+j, n-i-1}$ is freely spanned by the $(i+j)$-simplices of the form $(n-i-1,w, e<w_1<\dots <w_{i+j})$ with $w\in W$ and $0\leq n-i-1< n-\ell_{T}(w_{i+j})$.
 The boundary maps $E^0_{i,j}\rightarrow E^0_{i, j-1}$ on the $E^0$-page are induced by $\partial_k$ as defined in \eqnref{eq: boumapQ}, and the homology groups of these maps give rise to the $E^1$-page
\[ E^1_{i,j}=H_{i+j}(\mathcal{F}_i, \mathcal{F}_{i-1}), \]
which is the  relative integral homology of the pair $(\mathcal{F}_i, \mathcal{F}_{i-1})$. 

The following lemma describes the $E^1$-page.

\begin{lemma}\label{lem: Homoquo}
    For each $i=0,\dots, n-1$, we have the following isomorphisms
    \[ H_{i}(\mathcal{F}_i, \mathcal{F}_{i-1})\cong \mathbb{Z}W \otimes \widetilde{H}_{i-1}(\mathcal{L}_{[i]}) \cong \mathbb{Z}W \otimes d(\mathcal{B}_{i+1}),
    \]
    as free modules over the group ring $\mathbb{Z}W$, and for  each $k<i$ the homology group $H_{k}(\mathcal{F}_i, \mathcal{F}_{i-1})$ is trivial.
\end{lemma}
\begin{proof}
Recall that the $k$-th chain group of the pair $(\mathcal{F}_i, \mathcal{F}_{i-1})$ is 
\[C_k(\mathcal{F}_i, \mathcal{F}_{i-1})=C_k(\mathcal{F}_i)/ C_k(\mathcal{F}_{i-1})\cong C_{k, n-i-1}, \quad 0\leq k\leq i \]
 where $C_{k, n-i-1}$ is freely spanned by $k$-simplices of the form $(n-i-1,w, e<w_1<\dots<w_k)$ with $w\in W$ and  $0\leq n-i-1< n-\ell_{T}(w_k)$. Then we have the following chain complex which computes the homology of  the pair $(\mathcal{F}_i, \mathcal{F}_{i-1})$: 
\[     \begin{tikzcd}
    0 \arrow[r]& C_{i,n-1-i} \arrow[r,"\partial^0_{i}"] & C_{i-1,n-1-i} \arrow[r,"\partial^0_{i-1}"] & \cdots \arrow[r] & C_{0,n-1-i} \arrow[r] & 0,
  \end{tikzcd}  \]
where the boundary maps $\partial^0_k$  are induced from $\partial_k$ (see \eqref{eq: boumapQ}), that is,
\begin{align*}
  &\partial^0_k(n-1-i, w, e< w_1 < \dots < w_k)\\
=&\sum_{j=1}^k (-1)^j ( n-1-i,w, e< w_1< \dots < \widehat{w_j}<\dots < w_k ).
\end{align*}
 Note that for any $k$-simplex   in $C_{k,n-i-1}$, we have the condition  $n-1-i< n- \ell_{T}(w_k)$, i.e., $\ell_{T}(w_k)\leq i$. Therefore, $ w_1 < \dots < w_k$ is a $(k-1)$-chain in the rank-selected subposet $\mathcal{L}_{[i]}$ of $\mathcal{L}$.

 Denote by $C_{k-1}(\mathcal{L}_{[i]})$ the chain group over $\mathbb{Z}$ spanned by all $(k-1)$-chains in  $\mathcal{L}_{[i]}$. Then we have the $\mathbb{Z}W$-isomorphism
\begin{equation}\label{eq: isofil}
    C_{k, n-i-1} \cong \mathbb{Z}W \otimes C_{k-1}(\mathcal{L}_{[i]}),
\end{equation}
given by mapping the $k$-simplex $(n-1-i, w, e< w_1 < \dots < w_k)$ to the tensor $w\otimes (w_1 < \dots < w_k)$.
It is easily verified that this isomorphism is a chain map.  This gives rise to the $\mathbb{Z}W$-isomorphism of homology groups  
\[ H_{k}(\mathcal{F}_i,\mathcal{F}_{i-1})\cong  \mathbb{Z}W \otimes \widetilde{H}_{k-1}(\mathcal{L}_{[i]}), \quad 0\leq k\leq i.\]
 Since $\mathcal{L}$ is Cohen-Macaulay,  the  rank-selected subposet $\mathcal{L}_{[i]}$ is also Cohen-Macaulay for each $i$ \cite[Theorem 6.4]{Bac80}. Therefore, the relative homology groups $H_k(\mathcal{F}_i,\mathcal{F}_{i-1})$ are nontrivial if and only if $k=i$. In this case, using \thmref{thm: basisrksel} we have $d(\mathcal{B}_{i+1})\cong \widetilde{H}_{i-1}(\mathcal{L}_{[i]})$. Hence we have $H_{i}(\mathcal{F}_i, \mathcal{F}_{i-1}) \cong \mathbb{Z}W \otimes d(\mathcal{B}_{i+1})$ as free $\mathbb{Z}W$-modules. 
\end{proof}

 Recall from \thmref{thm: basisrksel} that the elements $z_{{\bf t}}=d(\beta_{{\bf t}}), {\bf t}\in \mathcal{D}_{[k-1]}$ are a basis for $d(\mathcal{B}_k)\cong \widetilde{H}_{k-2}(\mathcal{L}_{[k-1]})$. In particular, $d(\mathcal{B}_1)=\mathbb{Z}$. We have the following. 

\begin{theorem}\label{thm: fullcom}
 The integral homology of the  Milnor fibre $F_Q=Q^{-1}(1)$ is isomorphic to the homology of the following chain complex
   \[    \begin{tikzcd}
    0 \arrow[r]&  \mathbb{Z}W\otimes d(\mathcal{B}_n) \arrow[r,"\partial_{n-1}"]  & \cdots \arrow[r] & \mathbb{Z}W\otimes  d(\mathcal{B}_2)  \arrow[r,"\partial_1"]& \mathbb{Z}W\otimes d(\mathcal{B}_1) \arrow[r] & 0,
  \end{tikzcd}
   \]
   where the boundary maps  are given by 
   \begin{equation*}
     \partial_{k-1}(w\otimes z_{(t_1,\dots ,t_{k}}) )= \sum_{i=1}^{k} (-1)^{i-1} wt_i \otimes z_{(t_1^{t_i},\dots ,t_{i-1}^{t_i}, t_{i+1},\dots ,t_{k})}
   \end{equation*}
   for any $w\in W$ and  $(t_1,\dots ,t_k)\in \bigcup_{u\in \mathcal{L}_k} {\rm Rex}_{T}(u)$ with $2\leq k\leq n$. 
\end{theorem}

\begin{proof}
 We  use the spectral sequence associated to the filtration \eqref{eq: Ffiltration}. By \lemref{lem: Homoquo}, we have the $E^1$-page given by
 \[
  E^{1}_{i,j}=H_{i+j}(\mathcal{F}_i, \mathcal{F}_{i-1})\cong
\begin{cases}
   \mathbb{Z}W \otimes \widetilde{H}_{i-1}(\mathcal{L}_{[i]}), & \text{if\,} j=0, \\
   0, & \text{otherwise.}
\end{cases}
 \]
 Therefore, the  $E^1$-page is concentrated in a single row such that $E^1_{i,0}\cong \mathbb{Z}W \otimes \widetilde{H}_{i-1}(\mathcal{L}_{[i]})\cong \mathbb{Z}W\otimes  d(\mathcal{B}_{i+1})$.  It remains to determine the nonzero boundary maps on the $E^1$-page.

We first look at the boundary map on $\mathbb{Z}W \otimes \widetilde{H}_{i-1}(\mathcal{L}_{[i]})$ induced from  \eqref{eq: boumapQ}, and then translate it to $\mathbb{Z}W\otimes  d(\mathcal{B}_{i+1})$.  It follows from  definition \eqref{eq: boumapQ} and the isomorphism \eqnref{eq: isofil} that the boundary map $E^1_{i,0}\rightarrow E^{1}_{i-1,0}$ is given by the restriction of the linear map $\partial_{i}: \mathbb{Z}W \otimes C_{i-1}(\mathcal{L}_{[i]})\rightarrow  \mathbb{Z}W\otimes C_{i-2}(\mathcal{L}_{[i-1]} )$ defined by 
 \[ \partial_{i}(w\otimes (w_1< w_2 < \dots < w_i))= ww_1\otimes  (w_1^{-1}w_2< w_1^{-1}w_3< \dots < w_1^{-1} w_i) \]
  for any $w\in W$ and $w_j\in \mathcal{L}_{[i]}$ with $\ell_{T}(w_j)=j$ for $1\leq j\leq i$. For convenience, we identify the chain $w_1< w_2 < \dots < w_i$ with the sequence of reflections $(t_1,t_2, \dots,t_i)$, where $t_{j}=w_{j-1}^{-1}w_j$ for $1\leq j\leq i$. Then the above linear map can be written as 
\begin{equation}\label{eq: indboundmap}
   \partial_{i}(w\otimes (t_1, t_2, \dots, t_i))= wt_1\otimes  (t_2, \dots, t_i),
\end{equation}
which gives rise to the linear map $\partial_{i}: \mathbb{Z}W\otimes \mathcal{B}_i\rightarrow  \mathbb{Z}W\otimes \mathcal{B}_{i-1}$ for each $i$. 

Now we can restrict the above linear maps to $\mathbb{Z}W\otimes  d(\mathcal{B}_{i+1})$, which give the boundary maps of the chain complex  in the theorem. Recalling that $z_{(t_1, \dots,t_{i+1})}=d(\beta_{(t_1, \dots, t_{i+1})})\in  d(\mathcal{B}_{i+1})$, we have 
\begin{equation}\label{eq: Deltamap}
  \partial (w\otimes z_{(t_1,\dots ,t_{i+1}})) = (\partial \circ (1\otimes d)) (w\otimes \beta_{(t_1,\dots, t_{i+1})} ).
\end{equation}
 Using the definition \eqref{eq: theta} of  $d$, one can check that 
\[ \partial \circ (1\otimes d)= (1\otimes d)\circ \partial.  \]
Substituting this into the right hand side of \eqref{eq: Deltamap}, we have 
\begin{align*}
\partial (w\otimes z_{(t_1,\dots ,t_{i+1}})) &= (1\otimes d)\circ \partial( w\otimes \beta_{(t_1,\dots, t_{i+1})} ) \\
&= (1\otimes d)\circ \partial \bigg( w\otimes \sum_{j=1}^{i+1}(-1)^{j-1} (t_i, \beta_{( t_1^{t_j},\dots, t_{j-1}^{t_j},\, t_{j+1},\dots,t_{i+1})})  \bigg)\\
&= \sum_{j=1}^{i+1}(-1)^{j-1} (1\otimes d)( wt_i \otimes \beta_{( t_1^{t_j},\dots, t_{j-1}^{t_j},\, t_{j+1},\dots,t_{i+1})} )
\\
&= \sum_{j=1}^{i+1} (-1)^{j-1} wt_i \otimes z_{(t_1^{t_j},\dots ,t_{j-1}^{t_j}, t_{j+1},\dots ,t_{i+1})},
\end{align*}
where  the second equation follows from part (3) of \lemref{lem: betasum}, and the third equation  from \eqnref{eq: indboundmap}. This finishes the proof. 
\end{proof}

\begin{remark}
  The Coxeter group $W$ acts on the chain complex in \thmref{thm: fullcom} by left multiplication on the tensor factor $\mathbb{Z}W$,  and the monodromy action can be partly described as follows.  We may define  the cyclic group $\langle \gamma \rangle$ action on $\mathbb{Z}W\otimes \mathcal{B}_{k}$ by 
\begin{equation*}\label{eq: CoxactWA}
  \gamma. (w\otimes \beta_{t_1}\beta_{t_2}\cdots \beta_{t_k})= w\gamma \otimes \beta_{\gamma^{-1}t_1 \gamma}\beta_{\gamma^{-1}t_2 \gamma}\cdots \beta_{\gamma^{-1}t_k \gamma}, \quad w\in W, t_i\in T.
\end{equation*}
 This action commutes with the map $1\otimes d$. Hence the chain groups $\mathbb{Z}W\otimes d(\mathcal{B}_k)$ are all  $\langle \gamma \rangle$-modules. Straightforward calculation shows that the $\gamma$ action also commutes with the boundary map  of the chain complex given in \thmref{thm: fullcom}, and hence the homology groups $H_k(F_Q;\mathbb{Z})$ are $\langle \gamma \rangle$-modules.  Actually, the action of the Coxeter element $\gamma$ appears to be   the $n$-th power of the monodromy action on  the NCP model of $F_{Q}$. 
\end{remark}

Now we can derive a chain complex for $F_{Q_0}$ from that for $F_{Q}$. Denote by $\mathbb{Z}W_+$ (resp. $\mathbb{Z}W_{-}$) the free abelian subgroup of $\mathbb{Z}W$ spanned by even (resp. odd) elements of $W$. Let $\mathfrak{p}(k)$ be the parity function defined by $\mathfrak{p}(k)=+$ if $k$ is even and  $\mathfrak{p}(k)=-$ otherwise.

\begin{theorem}\label{thm: redfib}
 The integral homology of the  Milnor fibre $F_{Q_0}=Q_0^{-1}(1)$ is isomorphic to the homology of the following chain complex
   \[    \begin{tikzcd}
      0\longrightarrow \mathbb{Z}W_{\mathfrak{p}(n)}\otimes d(\mathcal{B}_n) \arrow[r,"\partial_{n-1}"]  & \cdots \arrow[r] & \mathbb{Z}W_{\mathfrak{p}(2)}\otimes  d(\mathcal{B}_2)  \arrow[r,"\partial_1"]& \mathbb{Z}W_{\mathfrak{p}(1)}\otimes d(\mathcal{B}_1)\longrightarrow 0,
  \end{tikzcd}
   \]
  where the boundary maps  are the same as those in \thmref{thm: fullcom}.
\end{theorem}
\begin{proof}
    We have $\mathbb{Z}W= \mathbb{Z}W_+\oplus \mathbb{Z}W_{-}$ as a  $\mathbb{Z}W_{+}$-module. Note that  if $w\in \mathbb{Z}W_{+}$ (resp. $w\in \mathbb{Z}W_-$) then $wt\in \mathbb{Z}W_-$ (resp. $w\in \mathbb{Z}W_+$) for any $t\in T$. Therefore, the chain complex in \thmref{thm: fullcom} splits into two isomorphic  chain complexes, which produce isomorphic homology groups. In particular, one of them has the chain groups $\mathbb{Z}W_{\mathfrak{p}(i)} \otimes d(\mathcal{B}_i)$ for $1\leq i\leq n$. This splitting  is  in accord with the disjoint union  $F_Q=Q_0^{-1}(1)\cup Q_0^{-1}(-1)$ of two homeomorphic connected components, which also have isomorphic homology groups.  By the above arguments, either of the two isomorphic chain complexes computes  the homology of $F_{Q_0}=Q_{0}^{-1}(1)$,   and we may choose the chain complex $(\mathbb{Z}W_{\mathfrak{p}(i)} \otimes d(\mathcal{B}_i), \partial_i)$.
\end{proof}

\begin{example}\label{exam: dihcon}(Dihedral group)
	Let $W$ be the dihedral group as considered in \exref{exam: dih}. We use the chain complex in \thmref{thm: redfib} to compute the homology of the corresponding Milnor fibre $F_{Q_0}$.
	
	Recall from  \exref{exam: dih} that $W$ has $m$ reflections $t_i=s_1(s_2s_1)^{i-1}, 1\leq i\leq m$, and $\mathcal{L}$  has  $m-1$ decreasing maximal chains which are labelled by elements of the set
	$ \mathcal{D}_{[1]}=\{(t_i,t_{i-1} )\mid 2\leq i\leq m \}$. By \thmref{thm: basisrksel}, the free abelian group $d(\mathcal{B}_2)=\widetilde{H}_{0}(\mathcal{L}_{[1]})$ has a basis comprising elements $z_{(t_i,t_{i-1})}$ for all  $(t_i,t_{i-1})\in \mathcal{D}_{[1]}$, and $d(\mathcal{B}_1)=\mathbb{Z}$. Note that $\mathbb{Z}W_+$ is spanned by all rotations $r_i:=(s_2s_1)^{i-1}, 1\leq i\leq m$, while $\mathbb{Z}W_-$ is spanned by all reflections $t_i, 1\leq i\leq m$. We have the following chain complex:
	  \[ \begin{tikzcd}
		0\arrow[r] & \mathbb{Z}W_{+}\otimes  d(\mathcal{B}_2)   \arrow[r,"\partial_1"]& \mathbb{Z}W_{-}\otimes d(\mathcal{B}_1) \arrow[r] & 0,
	\end{tikzcd} \]
	where the  boundary map $\partial_1$ is defined by
	\[ \partial_1( r_i \otimes z_{(t_j,t_{j-1})} )= r_i t_j\otimes 1 - r_i t_{j-1}\otimes 1, \]
	for $1\leq i\leq m$ and $2\leq j \leq  m$. If $r_i=1$, then the elements $ t_j- t_{j-1}, 2\leq j \leq m$ form a basis for the following  free abelian subgroup $\mathbb{Z}\widebar{W}_-$ of $\mathbb{Z}W_{-}$:
	\[ \mathbb{Z}\widebar{W}_{-}=\big\{ \sum_{i=1}^m \lambda_i t_i \mid \sum_{i=1}^m \lambda_i=0, \lambda_i\in \mathbb{Z} \big\}\cong \mathbb{Z}^{m-1}.  \]
	For any other $r_i$, we have  $r_it_j-r_it_{j-1}\in \mathbb{Z}\widebar{W}_-$ for  $2\leq j\leq m$. Therefore, ${\rm Im}\partial_1=\mathbb{Z}\widebar{W}_-\otimes \mathbb{Z}$. It follows that 
	  \[ H_0(F_{Q_0}; \mathbb{Z}) \cong \mathbb{Z}, \quad H_1(F_{Q_0}; \mathbb{Z})\cong \mathbb{Z}^{(m-1)^2}. \]
\end{example}

\subsection{The Milnor fibre of the discriminant of \texorpdfstring{$W$}{W}}
We refer to \cite{LT09} for the invariant theory for $W$. Recall that $W$ has a set of basic invariants $\{f_1, \dots, f_n\}$ of the $W$-invariant polynomial ring $\mathbb{C}[x_1, \dots, x_n]^W$.  Since $Q$ is $W$-invariant, there exists a unique polynomial $P(y_1, \dots, y_n)$ such that $Q=P(f_1, \dots, f_n)$. The polynomial $P$ is called the discriminant of $W$. The Milnor fibre $F_P:=P^{-1}(1)$ can be realised as the quotient $F_{Q}/W$. 

The NCP model of $F_P$ is also introduced in \cite{BFW18}. It is defined to be the $(n-1)$-dimensional CW-complex $\mathcal{F}^W$ whose $k$-cells are of the form 
\[
(m, e<w_1<\dots <w_k) \quad \text{with} \quad 0\leq m< n-\ell_{T}(w_k),
\]
where $e<w_1<\dots <w_k$ is a chain of $\mathcal{L}$. Note that $\mathcal{F}^W$ is not a simplicial complex, and each cell of $\mathcal{F}^W$ can be  obtained from a simplex of $\mathcal{F}$ by removing the group element $w$ in the middle of the triple (cf. \defref{def: ncpmodfull}).  

The homology groups of  $F_P$ can be calculated by using the NCP model $\mathcal{F}^W$. The boundary maps of $\mathcal{F}^W$ are  given by
\[
  \begin{aligned}
\partial_k(m, e< w_1 < \dots < w_k)= &(m+ \ell_T(w_1),  e < w_1^{-1}w_2< \dots < w_1^{-1}w_k) \\
 & +\sum_{i=1}^k (-1)^i ( m, e< w_1< \dots < \widehat{w_i}<\dots < w_k ).
\end{aligned}
\]
Since $F_P$ and  $\mathcal{F}^W$ have the same homotopy type, we obtain \[H_k(F_P;\mathbb{Z})\cong H_k(\mathcal{F}^W;\mathbb{Z}):= {\rm Ker}\, \partial_k/ {\rm Im}\, \partial_{k+1}.\]

We can introduce a filtration of $\mathcal{F}^W$ similar to \eqnref{eq: Ffiltration}. Precisely, the $i$-dimensional subcomplex $\mathcal{F}^W_i$ is defined to be the union of  maximal $i$-cells of the  form $(n-i-1, e<w_1<\dots< w_i)$ together with their faces, where  $w_1<\dots< w_i$ is a maximal chain of $\mathcal{L}_{[i]}$. Now we have the following lemma.

\begin{lemma}
  For each $i=0,\dots, n-1$, we have the following isomorphisms of free abelian groups:
    \[ H_{i}(\mathcal{F}^W_i, \mathcal{F}^W_{i-1})\cong  \widetilde{H}_{i-1}(\mathcal{L}_{[i]}) \cong  d(\mathcal{B}_{i+1}),
    \]
and for each $k<i$ the homology group $H_{k}(\mathcal{F}^W_i, \mathcal{F}^W_{i-1})$ is trivial.
\end{lemma}
\begin{proof}
The proof is analogous to that of \lemref{lem: Homoquo}.
\end{proof}

The following theorem is an improvement of \cite[Theorem 6.4]{BFW18}, which does not give an explicit description of  the chain groups and boundary maps.

\begin{theorem}\label{thm: disfibre}  
  The integral homology of $F_P$ is isomorphic to the homology of the following chain complex of free abelian groups:
   \begin{equation*}
  \xymatrix{
   0  \ar[r] & d(\mathcal{B}_n)  \ar[r]^-{\partial _{n-1}}  &  d(\mathcal{B}_{n-1}) \ar[r]& \cdots  \ar[r] &  d(\mathcal{B}_2) \ar[r]^-{\partial _1} & d(\mathcal{B}_1) \ar[r]&  0,
}
\end{equation*}
 where the boundary maps are given by 
 \begin{equation}\label{eq: bounddis}
   \partial _{k-1}(z_{(t_1,\dots ,t_{k}}))= \sum_{i=1}^{k} (-1)^{i-1} z_{(t_1^{t_i},\dots ,t_{i-1}^{t_i}, t_{i+1},\dots ,t_{k})}
 \end{equation}
for any $(t_1,\dots ,t_k)\in \bigcup_{w\in \mathcal{L}_k} {\rm Rex}_{T}(w)$ and $2\leq k\leq n$. In particular, $\partial_1=0$.
\end{theorem}
\begin{proof}
The proof is similar to that of \thmref{thm: fullcom}. Note that in particular $\partial _1=0$, since $\partial_1(z_{(t_1,t_2)})=z_{t_2}-z_{t_2t_1t_2}$ and $z_{t}=d(\beta_t)=1$ for any $t\in T$. 
\end{proof}

\begin{remark}
 Recall from \thmref{thm: basisrksel} that the elements $z_{{\bf t}}, {\bf t}\in \mathcal{D}_{[k-1]}$ are a basis for $d(\mathcal{B}_k)= \widetilde{H}_{k-2}(\mathcal{L}_{[k-1]})$. On the right hand side of equation \eqnref{eq: bounddis}, the elements $z_{(t_1^{t_i},\dots ,t_{i-1}^{t_i}, t_{i+1},\dots ,t_{k})}, 1\leq i\leq k$ are not $\mathbb{Z}$-linearly independent, and hence they are generally not basis elements of $d(\mathcal{B}_{k-1})$. 
\end{remark}

\begin{example}(Dihedral group) Let $W$ be a dihedral group. Recall from \exref{exam: dihcon} that $d(\mathcal{B}_2)\cong \mathbb{Z}^{m-1}$ and $d(\mathcal{B}_1)\cong \mathbb{Z}$. Note that we have  $\partial_1=0$ in the chain complex given in \thmref{thm: disfibre}. Therefore, the corresponding Milnor fibre $F_P$ has the homology groups $H_1(F_P;\mathbb{Z})\cong \mathbb{Z}^{m-1}$ and $H_0(F_P;\mathbb{Z})\cong \mathbb{Z}$.
\end{example}
\begin{example} (Type $A_3$)
 We continue  \exref{exam: A3con}, where $W={\rm Sym}_4$ and $\gamma=(1243)$.
	Then the chain complex given in \thmref{thm: disfibre} is 
	\[
		\xymatrix{
			0  \ar[r] & d(\mathcal{B}_3)  \ar[r]^-{\partial _{2}}  &  d(\mathcal{B}_2) \ar[r]^-{\partial _1} & d(\mathcal{B}_1) \ar[r]&  0,
		}
	\]
	where $\partial_1=0$ and for any $(t_1,t_2,t_3)\in {\rm Rex}_T(\gamma)$, 
	\[ \partial_2(z_{(t_1,t_2,t_3)})= z_{(t_2,t_3)}- z_{(t_1^{t_2},t_3)}+ z_{(t_1^{t_3},t_2^{t_3})}. \] 
 Recall from \exref{exam: A3con}  that $d(\mathcal{B}_3)= \widetilde{H}_{1}(\mathcal{L}_{[2]})$ (resp. $d(\mathcal{B}_2)= \widetilde{H}_{0}(\mathcal{L}_{[1]})$) has a basis consisting of elements $z_{{\bf t}}, {\bf t} \in \mathcal{D}_{[2]}$ (resp. $z_{{\bf t}}, {\bf t} \in \mathcal{D}_{[1]}$). Using these basis elements, we have 
	\begin{align*}
		\partial_2(z_{ ((14), (34), (12))})&= z_{((34),(12))}- z_{((13),(12))}+ z_{((24),(34))},\\
		\partial_2(z_{((13), (14), (12))}) & = z_{((14),(12))}- z_{((34),(12))} + z_{((23),(24))},\\
		\partial_2(z_{((24), (14), (34))} )& = z_{((14),(34))}- z_{((12),(34))} + z_{((23),(13))}\\
		&= z_{((14),(12))}+z_{((34),(12))} -z_{((13),(12))}+z_{((24),(34))}+ z_{((23),(24))}, \\
		\partial_2( z_{((13), (24), (14))} )& = z_{((24),(14))}- z_{((13),(14))} + z_{((34),(12))}\\
		&=z_{((24),(34))}+ z_{((34).(12))}-z_{((13),(12))}+z_{((34),(12))},\\
		\partial_2(z_{((23), (13), (24))}) & = z_{((13),(24))}- z_{((12),(24))} + z_{((34),(13))}\\
		&=z_{((34),(12))}.
	\end{align*}
It is easy to obtain the integral homology of the Milnor fibre $F_P$: 
\[ H_{2}(F_P;\mathbb{Z})\cong \mathbb{Z}^2, \quad H_1(F_P;\mathbb{Z})\cong \mathbb{Z}^2, \quad  H_0(F_P;\mathbb{Z})\cong \mathbb{Z}. \]
\end{example}

\section{Connection with the hyperplane complement}\label{sec: app2}

Recall that $M= \mathbb{C}^n - \bigcup_{t\in T} H_t$ is the hyperplane complement of $W$.
Using the algebra $\mathcal{B}$,  we give a chain complex whose homology is isomorphic to the homology of the hyperplane complement $M$. In addition, the integral homology of the Artin group $A(W)$ can also be calculated by a chain complex in terms of $\mathcal{B}$.

\subsection{The NCP model of \texorpdfstring{$M$}{M}} 
In \cite[Example 2.3]{BFW18}, the NCP model of $M$ is defined to be the $n$-dimensional CW complex $\mathcal{M}$,   whose $k$-cells  are of the form $(w, e< w_1< \dots <w_k)$, where $w\in W$ and $e< w_1< \dots <w_k$ is a chain of $\mathcal{L}$. Note that $w_k$ can be the maximal element $\gamma$ of $\mathcal{L}$. Each cell has $k$ boundary faces of the form $(w, e< w_1< \dots <\hat{w}_i<\dots< w_k)$ for $1\leq i\leq k$, and the remaining boundary face  $(ww_1, e< w_1^{-1}w_2< \dots < w_{1}^{-1}w_k)$.

We can compute the homology of $\mathcal{M}$.  The boundary maps $\partial_k$ are given by 
\begin{equation}\label{eq: boundmap}
  \begin{aligned}
  \partial_k((w,e<w_1<\dots <w_k))= &\, \, (ww_1,e<w_1^{-1}w_2<\dots< w_1^{-1}w_k) \\
  &+\sum_{i=1}^k (-1)^i \,  (w, e<w_1<\dots<\widehat{w_i}<\dots <w_k ).
\end{aligned}
\end{equation}
Since $\mathcal{M}$ has the homotopy type of the hyperplane complement $M$, we have 
\[
 H^k(M;\mathbb{Z})\cong H^k(\mathcal{M}; \mathbb{Z}): ={\rm Ker}\, \partial_k/ {\rm Im}\, \partial_{k+1}, \quad 0\leq k\leq n.
\]

Now we define a filtration of $\mathcal{M}$. Let $\mathcal{M}_i$ be the $i$-dimensional subcomplex of $\mathcal{M}$ whose $k$-cells  are of the form \
\[ (w, e<w_1<\dots<w_k), \quad \text{with\,  $w\in W$ and $\ell_{T}(w_k)\leq i$.}\]
Then we obtain an increasing filtration of $\mathcal{M}$ by  subcomplexes $\mathcal{M}_i$ for $0\leq i\leq n$.

\begin{lemma}\label{lem: hpcomE1}
   For each $1\leq i\leq n$, we have the isomorphisms
      \[ H_i(\mathcal{M}_i, \mathcal{M}_{i-1}) \cong \bigoplus_{w
    \in \mathcal{L}_i} \mathbb{Z}W \otimes \widetilde{H}_{i-2}(e,w) \cong  \mathbb{Z}W\otimes \mathcal{B}_i\]
    as free $\mathbb{Z}W$-modules, and  for each $k<i$ the homology group $H_k(\mathcal{M}_i, \mathcal{M}_{i-1})$  is trivial.
\end{lemma}
\begin{proof}
 Denote by $C_{k}(\mathcal{M}_i)$ the $k$-th chain group of $\mathcal{M}_i$. Then  the  $k$-th chain group of the pair $(\mathcal{M}_i, \mathcal{M}_{i-1})$ is 
   $ C_{k}(\mathcal{M}_i, \mathcal{M}_{i-1}):=C_{k}(\mathcal{M}_i)/ C_{k}(\mathcal{M}_{i-1})$.  
  This is isomorphic to the  abelian group freely spanned by all $k$-cells of the form $(w, e<w_1<\dots <w_k)$,  where $w\in W$ and $e<w_1<\dots <w_k$ is a chain of $\mathcal{L}$ with $\ell_T(w_k)=i$. This gives rise to the following isomorphism of $\mathbb{Z}W$-modules: 
  \begin{equation}\label{eq: iso}
    C_{k}(\mathcal{M}_i, \mathcal{M}_{i-1})\cong \mathbb{Z}W\otimes \bigoplus_{u\in \mathcal{L}_i} C_{k-2}(e,u), 
  \end{equation}
  where $C_{k-2}(e,u)$ is the $(k-2)$-th chain group in the open interval $(e,u)$ of $\mathcal{L}$. More precisely, letting $[w, e<w_1<\dots <w_k]$ be the quotient  image of the $k$-cell $(w, e<w_1<\dots <w_k)$ in $C_{k}(\mathcal{M}_i, \mathcal{M}_{i-1})$,  the isomorphism above is given by mapping each $[w, e<w_1<\dots <w_k]$ to the tensor $w \otimes (w_1<\dots<w_{k-1})$ and extend it linearly. Clearly, this is an isomorphism of $\mathbb{Z}W$-modules. 

 Note that the  boundary maps $\partial^0_k$ on $C_{k}(\mathcal{M}_i, \mathcal{M}_{i-1})$ are induced from \eqnref{eq: boundmap}. Precisely,  the boundary maps $\partial^0_k$ are given by 
  \[ \partial^0_k([w, e<w_1<\dots <w_k])=\sum_{i=1}^{k-1} (-1)^i  \,  [w, e<w_1<\dots<\widehat{w_i}<\dots <w_k ].  \]
  Note that on the right hand side $w_k$ should  not be omitted.  
 This passes to boundary maps on  $\mathbb{Z}W\otimes \bigoplus_{u\in \mathcal{L}_i} C_{k-2}(e,u)$ through the isomorphism \eqref{eq: iso}, i.e. we have 
\[ \partial^0_k( w\otimes (w_1<\dots < w_{k-1}))=\sum_{i=1}^{k-1}(-1)^i w\otimes (w_1<\dots<\widehat{w_i}<\dots <w_{k-1}).   \]
In particular, the isomorphism \eqref{eq: iso} is a chain map, and hence for each $k$ we have 
  \[ H_{k}(\mathcal{M}_i, \mathcal{M}_{i-1})\cong \mathbb{Z}W\otimes \bigoplus_{u\in \mathcal{L}_i} \widetilde{H}_{k-2}(e,u). \]
By the Cohen-Macaulay property of $\mathcal{L}$, the right hand side is nontrivial if and only if $k=i$. In this case, by \thmref{thm: homint} we have $\bigoplus_{u\in \mathcal{L}_i} \widetilde{H}_{k-2}(e,u)\cong \bigoplus_{u\in \mathcal{L}_i}\mathcal{B}_u= \mathcal{B}_i$ for $1\leq i\leq n$. Hence $H_{k}(\mathcal{M}_i, \mathcal{M}_{i-1}) \cong \mathbb{Z}W\otimes \mathcal{B}_i$ as $\mathbb{Z}W$-modules. 
\end{proof}

\begin{theorem}\label{thm: homoM}
The integral homology of the hyperplane complement $M$ is isomorphic to the homology of the following chain  complex of abelian groups:
  \[ \begin{tikzcd}
    0 \arrow[r]&  \mathbb{Z}W\otimes  \mathcal{B}_n \arrow[r,"\partial_{n}"]  & \cdots \arrow[r] &   \mathbb{Z}W\otimes \mathcal{B}_1  \arrow[r,"\partial_1"]&  \mathbb{Z}W\otimes \mathcal{B}_0 \arrow[r] & 0,
  \end{tikzcd} \]
 where the boundary maps are given by 
\[ 
  \partial_k(w\otimes \beta_{(t_1, t_2, \dots ,t_k)})= \sum_{i=1}^k(-1)^{i-1} wt_i\otimes \beta_{(t_1^{t_i}, \dots, t_{i-1}^{t_i},t_{i+1}, \dots, t_k)}
  - \sum_{i=1}^k (-1)^{i-1} w\otimes \beta_{(t_1, \dots, \hat{t}_i, \dots, t_k)} 
\]
for any $w\in W$ and $(t_1, t_2, \dots, t_k)\in \bigcup_{u\in \mathcal{L}_k} {\rm Rex}_T(u)$ with $1\leq k\leq n$.
\end{theorem}
\begin{proof}
 We apply the spectral sequence to the filtration $\emptyset=\mathcal{M}_{-1} \subset \mathcal{M}_0 \subset \mathcal{M}_1\subset \dots \subset \mathcal{M}_n=\mathcal{M}$. By \lemref{lem: hpcomE1}, the $E^1$-page is concentrated in a single row, which is the chain complex 
  \[ \begin{tikzcd}
    0 \arrow[r]&  H_n(\mathcal{M}_n, \mathcal{M}_{n-1}) \arrow[r,"\partial_{n}"]  & \cdots \arrow[r] &   H_1(\mathcal{M}_1, \mathcal{M}_0)  \arrow[r,"\partial_1"]&  H_0(\mathcal{M}_0) \arrow[r] & 0,
  \end{tikzcd} \]
where $H_i(\mathcal{M}_i, \mathcal{M}_{i-1})\cong \mathbb{Z}W\otimes \mathcal{B}_i$ for $1\leq i\leq n$, and by definition $H_0(\mathcal{M}_0)\cong \mathbb{Z}W \cong \mathbb{Z}W \otimes \mathcal{B}_0$.

It remains to determine the boundary maps, which are induced from \eqnref{eq: boundmap}. Note that $x\in H_i(\mathcal{M}_i, \mathcal{M}_{i-1})$ is a sum of cells of the form $[w, e< w_1<\dots < w_i]$ with $\ell_T(w_j)=j$ for $1\leq j\leq i$. Using \eqnref{eq: boundmap}, we obtain the induced boundary map defined on those cells by 
\[
\begin{aligned}
\partial_i([w, e< w_1<\dots < w_i]) =&  [ww_1, e< w_1^{-1}w_2< \dots < w_1^{-1}w_i] \\
 &+ (-1)^k [w, e< w_1< \dots < w_{i-1}]. 
\end{aligned}
\]
We may identify $e< w_1<\dots < w_i$ with the sequence $(t_1, \dots , t_i)$, where $t_j= w_{j-1}^{-1}w_j$ for $1\leq j\leq i$. Then the above map can be rewritten as 
\begin{equation}\label{eq: relbound}
\partial_i([w, (t_1, \dots , t_i)]) =  [wt_1, (t_2, \dots , t_i)]+ (-1)^k [w, (t_1, \dots , t_{i-1})].     
\end{equation}

We proceed to write $x\in H_i(\mathcal{M}_i, \mathcal{M}_{i-1})$ explicitly in terms of elements $[w, (t_1, \dots , t_i)]$. For notational convenience, we assume  $\sum_{k}\lambda_k [w, (t_1, \dots , t_i)]= \sum_k[w, \lambda_k (t_1, \dots , t_i)] $ for $\lambda_k\in \mathbb{Z}$. By \lemref{lem: hpcomE1} and  \thmref{thm: homint}, we have the isomorphisms
\begin{equation}\label{eq: isorelhom}
     \mathbb{Z}W\otimes \mathcal{B}_i \cong \mathbb{Z}W\otimes  \bigoplus_{u\in \mathcal{L}_i} \widetilde{H}_{k-2}(e,u)\cong H_i(\mathcal{M}_i,\mathcal{M}_{i-1}),
\end{equation}
where each element $w\otimes \beta_{(t_1, \dots , t_i)}$ is sent to $w\otimes d(\beta_{(t_1, \dots , t_i)}) $, and then to $[w, \beta_{(t_1, \dots , t_k)}]$. Therefore, $x\in H_i(\mathcal{M}_i, \mathcal{M}_{i-1})$  is a linear combination of elements $[w, \beta_{(t_1, \dots , t_k)}]\in H_i(\mathcal{M}_i,\mathcal{M}_{i-1})$. Now applying the map \eqnref{eq: relbound} to the element $[w, \beta_{(t_1, \dots , t_k)}]$ and using \lemref{lem: betasum}, we have 
\[ \partial_k([w, \beta_{(t_1, \dots , t_k)}])= \sum_{i=1}^k(-1)^{i-1} [wt_i,  \beta_{(t_1^{t_i}, \dots, t_{i-1}^{t_i},t_{i+1}, \dots, t_k)}] - \sum_{i=1}^k (-1)^{i-1} [w,\beta_{(t_1, \dots, \hat{t}_i, \dots, t_k)}]. \]
Going backwards along the isomorphisms \eqnref{eq: isorelhom}, we obtain the boundary maps on $\mathbb{Z}W\otimes \mathcal{B}_i$ as given in the theorem.
\end{proof}

\begin{remark}
  It is well known that the cohomology ring $H^*(M;\mathbb{Z})$ is isomorphic to the Orlik-Solomon algebra \cite{OS80}. In view of \thmref{thm: homoM}, there may be a mysterious link between $\mathbb{Z}W\otimes \mathcal{B}$ and the Orlik-Solomon algebra. 
\end{remark}

\subsection{The NCP model of \texorpdfstring{$M/W$}{M/W}}
Recall that the hyperplane complement $M$ is a $K(\pi,1)$ space; its fundamental group is the pure Artin group $PB(W)$ of type $W$ and all higher homotopy groups are trivial. Since $M$ admits a free $W$-action, the quotient space $M/W$ is also a $K(\pi, 1)$ with the fundamental group being the Artin group $A(W)$ of type $W$.

In \cite[Example 2.2]{BFW18}, it is shown that $M/W$ is homotopy equivalent to the NCP model $\mathcal{M}^W$, which is a CW complex whose $k$ cells are the chains $e<w_1<\dots <w_k$ of the NCP lattice $\mathcal{L}$. Any such $k$-cell has $k$ maximal faces $e<w_1<\dots < \hat{w}_i< \dots <w_k $ for $1\leq i\leq k$, and the remaining maximal face $e< w_1^{-1}w_2< \dots < w_1^{-1} w_k$. One can define the boundary maps $\partial_k$ of $\mathcal{M}^W$ in a similar way to \eqnref{eq: boundmap}. Since $M/W$ is a $K(\pi,1)$, we have 
\[ H_k(M/W;\mathbb{Z})= H_k(A(W); \mathbb{Z})\cong H_{k}(\mathcal{M}^W; \mathbb{Z})= {\rm Ker}\, \partial_k/ {\rm Im}\, \partial_{k+1}.  \]
Therefore, the integral homology of the Artin group $A(W)$ can also be calculated by the NCP model $\mathcal{M}^W$. 

We can define a filtration of $\mathcal{M}^W$ similar to that of $\mathcal{M}$. Precisely, denote by $\mathcal{M}_i^W$ the $i$-dimensional subcomplex of $\mathcal{M}$ whose $k$-cells are the chains $e< w_1<\dots <w_k$ of $\mathcal{L}$ with $\ell_T(w_k)\leq i$.  
Then we have the following lemma.

\begin{lemma}\label{lem: homrelMquo}
 For each $1\leq i\leq n$, we have the isomorphisms of free abelian groups 
      \[ H_i(\mathcal{M}^W_i, \mathcal{M}^W_{i-1}) \cong \bigoplus_{w
    \in \mathcal{L}_i} \widetilde{H}_{i-2}(e,w) \cong  \mathcal{B}_i, \]
    and  for each $k<i$ the homology group $H_k(\mathcal{M}^W_i, \mathcal{M}^W_{i-1})$ is trivial.
\end{lemma}
\begin{proof}
 The proof is similar to that of \lemref{lem: hpcomE1}.
\end{proof}

\begin{theorem}\label{thm: homoMquo}
The integral homology of  $M/W$ or the Artin group $A(W)$ is isomorphic to the homology of the following chain  complex of abelian groups:
  \[ \begin{tikzcd}
    0 \arrow[r]&    \mathcal{B}_n \arrow[r,"\partial_{n}"]  & \mathcal{B}_{n-1}\arrow[r]& \cdots \arrow[r] &    \mathcal{B}_1  \arrow[r,"\partial_1"]&   \mathcal{B}_0 \arrow[r] & 0,
  \end{tikzcd} \]
 where the boundary maps are given by 
\begin{equation}\label{eq: HomArtin}
  \partial_k( \beta_{(t_1, t_2, \dots ,t_k)})= \sum_{i=1}^k(-1)^{i-1}  \beta_{(t_1^{t_i}, \dots, t_{i-1}^{t_i},t_{i+1}, \dots, t_k)}
  - \sum_{i=1}^k (-1)^{i-1}  \beta_{(t_1, \dots, \hat{t}_i, \dots, t_k)} 
\end{equation}
for $(t_1, t_2, \dots, t_k)\in \bigcup_{u\in \mathcal{L}_k} {\rm Rex}_T(u)$ with $1\leq k\leq n$.
\end{theorem}
\begin{proof}
  The proof is similar  to that of \thmref{thm: homoM}, by using \lemref{lem: homrelMquo}. 
\end{proof}

\begin{remark}
 It follows from \thmref{thm: homint} that the elements $\beta_{{\bf t}}, {\bf t}\in \bigcup_{w\in \mathcal{L}_k} \mathcal{D}_w$ form a basis of $\mathcal{B}_k$. On the right hand side of equation \eqref{eq: HomArtin}, for any $(t_1, \dots t_k)\in \mathcal{D}_w$ the elements $\beta_{(t_1, \dots, \hat{t}_i, \dots, t_k)}, 1\leq i\leq k$ are $\mathbb{Z}$-linearly independent, while $\beta_{(t_1^{t_i}, \dots, t_{i-1}^{t_i},t_{i+1}, \dots, t_k)}, 1\leq i\leq k$ are generally not  basis elements of $\mathcal{B}_{k-1}$. However,  we can use quadratic relations given in \propref{prop: quadrel} to express $\beta_{(t_1^{t_i}, \dots, t_{i-1}^{t_i},t_{i+1}, \dots, t_k)}$ in terms of the basis elements  of $\mathcal{B}_{k-1}$.  As we have mentioned in \rmkref{rmk: quadrel}, the quadratic relations actually generate all linear relations among elements $\beta_{{\bf t}}$.
\end{remark}

\begin{example}
	(Dihedral group) We continue \exref{exam: dihbasis} and compute the homology of the Artin group of type $I_2(m)$. Recall that $\mathcal{B}_2$ has a basis consisting of elements $\beta_{(t_i,t_{i-1})}, 2\leq i\leq m$, where $t_i=s_1(s_2s_1)^{i-1}$ with the index taken modulo $m$, and $\mathcal{B}_1$ is spanned by the basis $\beta_{t_i}, 1\leq i\leq m$.  We have the following chain complex:
	\[\begin{tikzcd}
		0 \arrow[r]&    \mathcal{B}_2 \arrow[r,"\partial_{2}"]   &    \mathcal{B}_1  \arrow[r,"\partial_1"]&   \mathcal{B}_0 \arrow[r] & 0,
	\end{tikzcd}\]
where $\partial_1=0$ and for $2\leq i\leq m$
\[
 \partial_{2}(\beta_{(t_i,t_{i-1})})= \beta_{t_i}- \beta_{t_{i}^{t_{i-1}}}=\beta_{t_i}- \beta_{t_{i-2}}.
\]
Note that if $m$ is even, we have the relation $\sum_{i=0}^{m/2-1} (\beta_{t_{m-2i}}-\beta_{t_{m-2(i+1)}})=0$. It is easily verified that ${\rm Im}\, \partial_2 \cong \mathbb{Z}^{m-2}$ if $m$ is even, and ${\rm Im}\, \partial_2 \cong \mathbb{Z}^{m-1}$ if $m$ is odd. Therefore, the integral homology groups $H_k$  are given by 
\[
\begin{aligned}
  H_0\cong \mathbb{Z}, &\quad  H_1\cong \mathbb{Z}, &\;  H_2\cong 0, \quad  & \text{if $m$ is odd,}\\
H_0\cong \mathbb{Z}, &\quad H_1\cong \mathbb{Z}^2, &\;  H_2\cong \mathbb{Z}, \quad & \text{if $m$ is even.}
\end{aligned}
\]

\end{example}

\subsection{Some remarks}
  The integral cohomology of the Artin group $A(W)$ has been obtained using various methods. This was initiated by Arnol'd in  \cite{Arn70}, where he computed the integral cohomology groups of the braid group $B_n$ (i.e. Artin group of type $A_{n-1}$) for $n\leq 11$ using the space of square-free complex polynomials of degree $n$. Subsequently, Fuks \cite{Fuk70} computed the cohomologies of the braid groups with coefficients in $\mathbb{Z}/2\mathbb{Z}$,  using the  one-point compactification of the configuration space. This method was later applied by  Vainshtein \cite{Vai78}, who gave a complete description of the cohomologies of the braid groups with coefficients in $\mathbb{Z}$ and $\mathbb{Z}/p\mathbb{Z}$ for any prime number $p$.   In \cite{Gor81}, Goryunov used the methods of Fuks and Vainshtein to express the integral cohomologies of the Artin groups of type $C_n$ and $D_n$ in terms of the cohomologies of the braid groups. Finally, for all exceptional types,  the  integral cohomologies of the Artin groups were computed by Salvetti through the  Salvetti complex \cite{Sal94}.

\thmref{thm: homoMquo} provides a unified approach to computing the integral  homology of the Artin group $A(W)$ for any finite Coxeter group $W$. By computer calculations we obtain integral homology groups  $H_*(A(W); \mathbb{Z})$ for  types $A_n (n\leq 8)$, $B_n/C_n (n\leq 6)$, $D_n (n\leq 7)$ and  exceptional types $H_3, H_4, F_4, E_6$. These homology groups can be translated into cohomology groups by the universal coefficient theorem, and the results coincide with those given in \cite{Arn70}, \cite{Gor81} and  \cite[Table 1]{Sal94}. Recently, Paolini and Salvetti \cite{PS21} proved the $K(\pi,1)$ conjecture for affine Artin groups, using the EL-shellability of the affine noncrossing partition posets. We hope to extend techniques in this paper to affine Artin groups,  and give a unified method to compute the corresponding integral (co)homology. The cohomology of affine Artin groups of types $\widetilde{A}_n$ and $\widetilde{B}_n$ were computed in \cite{CMS08,CMS10}. 

Compared with the Salvetti complex, our chain complex in \thmref{thm: homoMquo} has larger ranks. For example, the rank of the top chain group $\mathcal{B}_n$ equals $\mu(\mathcal{L})$, which in type $A_n$ is given by the Catalan number $c_n=\frac{1}{n+1}\binom{2n}{n}$ (see \cite[Section 4.8]{Zha20}). Despite the large ranks, for the Milnor fibres our chain complexes in \thmref{thm: redfib} and \thmref{thm: disfibre}   reduce significantly the calculation of the integral homology to a mechanical computation in terms of the algebra $\mathcal{B}$, while the Salvetti complex works over the ring $\mathbb{Z}[q,q^{-1}]$ \cite{CS04}, which usually makes it difficult to obtain the abelian group structure of the homology. Moreover, it is expected that our approach would compute efficiently  the (co)homology of the Milnor fibre with coefficients in an irreducible representation of $\mathbb{C}W$  \cite[Appendix A]{LZ22}. We hope to use the method developed in this paper and \cite{LZ22}  to prove the representation stability result in the sense of \cite{CF13} for the Milnor fibre. In type $A_{n}$, the representation stability was obtained in \cite{MT21}.

\appendix

\section{Computational results}\label{sec: appd}
\subsection{Exceptional types}
All calculations in this appendix are done with the  Magma computational algebra system except for type $I_2(m)$. By computer calculations, we obtain $H_i(F_P; \mathbb{Z})$ for all exceptional types in Table \ref{tab: disexc}, where $\mathbb{Z}_n$ denotes the quotient $\mathbb{Z}/n \mathbb{Z}$ for any integer $n$.  A similar table appears in \cite[Table 1]{CS04} with entries being ideals in $\mathbb{Z}[q,q^{-1}]$. Our table here exhibits clearer abelian group structures.

\begin{table}[H]
\begin{tabular}{c|cccccccc} \hline
         & $H_0(F_P)$         & $H_1(F_P)$   & $H_2(F_P)$  & $H_3(F_P)$  & $H_4(F_P)$   & $H_5(F_P)$   & $H_6(F_P)$   & $H_7(F_P)$            \\ \hline
$I_2(m)$ & $\mathbb{Z}$  & $\mathbb{Z}^{m-1}$   &     &        &         &         &      &          \\ 
$H_3$    & $\mathbb{Z}$   & $0$   & $\mathbb{Z}^7$    &        &         &         &       &               \\ 
$H_4$    & $\mathbb{Z}$   & $0$   & $\mathbb{Z}_2$   & $\mathbb{Z}^{43}$   &         &    &      &      \\ 
$F_4$    & $\mathbb{Z}$   & $\mathbb{Z}$  & $\mathbb{Z}^5$   & $\mathbb{Z}^{15}$   &    &    &     &       \\ 
$E_6$    & $\mathbb{Z}$   & $0$  & $\mathbb{Z}_2$  & $0$  & $\mathbb{Z}^{6}$   & $\mathbb{Z}^{14}$   &         &                  \\ 
$E_7$    & $\mathbb{Z}$   & $0$  & $\mathbb{Z}_2$  & $\mathbb{Z}_2$ & $\mathbb{Z}^2\oplus  \mathbb{Z}_6$  & $\mathbb{Z}^2$  & $\mathbb{Z}^{15}$   &                 \\ 
$E_8$    & $\mathbb{Z}$   & $0$  & $\mathbb{Z}_2$  & $0$ & $ \mathbb{Z}^2\oplus \mathbb{Z}_6$  & $\mathbb{Z}_3$  & $\mathbb{Z}^8\oplus \mathbb{Z}_2$  & $\mathbb{Z}^{55}$        \\ \hline
\end{tabular}
\vspace*{2mm}
\caption{Integral homology groups of $F_P$ for exceptional types}
\label{tab: disexc}
\end{table}

In Table \ref{tab: redexc}, we  list the integral homology groups of the Milnor fibre $F_{Q_0}$ for Coxeter groups of types $I_2(m)$, $H_3$ and $F_4$. These homology groups are all torsion free. In \cite[Table 1]{Set09} Settepanella calculated $H_*(F_{Q_0};\mathbb{Q})$ for types $H_3,H_4,F_4$. However, her results for $H_4$ and $F_4$ do not meet the property that the Euler characteristic of $F_{Q_0}$ should be divisible by the number of hyperplanes of $W$; see \cite[Proposition 3]{DL16}. 

\begin{table}[H]
\centering

{\def\arraystretch{1.6}\tabcolsep=10pt
\begin{tabular}{c|cccc}
\hline
      & $H_0(F_{Q_0})$ & $H_1(F_{Q_0})$  & $H_2(F_{Q_0})$  & $H_3(F_{Q_0})$   \\ \hline
$I_2(m)$ & $\mathbb{Z}$     & $\mathbb{Z}^{(m-1)^2}$      &  &     \\
$H_3$ & $\mathbb{Z} $    & $\mathbb{Z}^{14}$   &  $\mathbb{Z}^{493}$     &    \\

$F_4$ & $\mathbb{Z}$     & $\mathbb{Z}^{23}$    & $\mathbb{Z}^{183}$    & $\mathbb{Z}^{5921} $       \\ \hline

\end{tabular}
}
\vspace*{2mm}
\caption{Integral homology groups of $F_{Q_0}$ for exceptional types}
\label{tab: redexc}
\end{table}

\subsection{Infinite families in low ranks}
The integral homology groups $H_{*}(F_P;\mathbb{Z})$ in types $A_n$, $B_n$ and $D_n$ for $n\leq 8$  are given in Table \ref{tab: disA}, Table \ref{tab: disB} and Table \ref{tab: disD}, respectively. Table \ref{tab: disA} can be found in \cite[Table 4]{Cal06}.  Callegaro also determines the homology ring for $W$ of type $B_n$ \cite{Cal07}. However, due to the complicated ring structure we are not able to obtain explicit torsion in homology from its description. As far as we know, the results for type $D_n$ are completely new.

\begin{table}[H]
\begin{tabular}{c|cccccccc} \hline
         & $H_0$          & $H_1$             & $H_2$           & $H_3$  & $H_4$   & $H_5$   & $H_6$   & $H_7$             \\ \hline
 $A_1$     & $\mathbb{Z}$ &                   &                 &        &         &         &         &            \\
$A_2$    & $\mathbb{Z}$   & $\mathbb{Z}^{2}$  &             &        &         &         &      &                   \\ 
$A_3$    & $\mathbb{Z}$   & $\mathbb{Z}^2$    & $\mathbb{Z}^2$  &        &         &         &       &                 \\ 
$A_4$    & $\mathbb{Z}$   & $0$               & $\mathbb{Z}^2$  & $\mathbb{Z}^{4}$   &         &    &      &      \\ 
$A_5$    & $\mathbb{Z}$   & $0$               & $ \mathbb{Z}^2\oplus \mathbb{Z}_2$  & $\mathbb{Z}^{4}$   & $\mathbb{Z}^{2}$    &   &    &       \\ 
$A_6$    & $\mathbb{Z}$   & $0$               & $\mathbb{Z}^2 \oplus  \mathbb{Z}_2 $  & $0$  & $\mathbb{Z}^{2}$   & $\mathbb{Z}^{6}$    &         &                \\ 
$A_7$    & $\mathbb{Z}$   & $0$               & $\mathbb{Z}_2$  & $0$ & $ \mathbb{Z}^2\oplus \mathbb{Z}_3 $  & $\mathbb{Z}^6$  & $\mathbb{Z}^4$  &                \\ 
$A_8$    & $\mathbb{Z}$   & $0$               & $\mathbb{Z}_2$  & $\mathbb{Z}^2$ & $ \mathbb{Z}^2\oplus \mathbb{Z}_3$  & $\mathbb{Z}_2^2$  & $\mathbb{Z}^4$  & $\mathbb{Z}^6$         \\ \hline
\end{tabular}
\vspace*{2mm}
\caption{Integral homology groups of $F_P$ for type $A_n, n\leq 8$}
\label{tab: disA}
\end{table}

\begin{table}[H]
\begin{tabular}{c|cccccccc} \hline
         & $H_0$         & $H_1$   & $H_2$  & $H_3$  & $H_4$   & $H_5$   & $H_6$   & $H_7$            \\ \hline
$B_2$    & $\mathbb{Z}$  & $\mathbb{Z}^{3}$   &     &        &         &         &      &           \\ 
$B_3$    & $\mathbb{Z}$   & $\mathbb{Z}$   & $\mathbb{Z}^3$    &        &         &         &       &               \\ 
$B_4$    & $\mathbb{Z}$   & $\mathbb{Z}$   & $\mathbb{Z}^3$   & $\mathbb{Z}^{7}$   &         &    &      &      \\ 
$B_5$    & $\mathbb{Z}$   & $\mathbb{Z}$  & $ \mathbb{Z}\oplus \mathbb{Z}_2$   & $\mathbb{Z}$   & $\mathbb{Z}^5$   &    &     &       \\ 
$B_6$    & $\mathbb{Z}$   & $\mathbb{Z}$  & $ \mathbb{Z}\oplus \mathbb{Z}_2$  & $\mathbb{Z}^3$  & $\mathbb{Z}^5$   & $\mathbb{Z}^9$  &         &                  \\ 
$B_7$    & $\mathbb{Z}$   & $\mathbb{Z}$  & $ \mathbb{Z}\oplus \mathbb{Z}_2$  & $ \mathbb{Z}\oplus \mathbb{Z}_2$ & $ \mathbb{Z}\oplus \mathbb{Z}_6 $  & $\mathbb{Z}$  & $\mathbb{Z}^7$   &                 \\ 
$B_8$    & $\mathbb{Z}$   & $ \mathbb{Z} $  & $ \mathbb{Z} \oplus \mathbb{Z}_2 $  &$ \mathbb{Z}\oplus \mathbb{Z}_2 $ & $\mathbb{Z}^3\oplus \mathbb{Z}_6 $ & $ \mathbb{Z}^3 \oplus\mathbb{Z}_3$  & $\mathbb{Z}^7$  & $\mathbb{Z}^{15}$          \\ \hline
\end{tabular}
\vspace*{2mm}
\caption{Integral homology groups of $F_P$ for type  $B_n,n\leq 8$ }
\label{tab: disB}
\end{table}

\begin{table}[H]
\begin{tabular}{c|cccccccc} \hline
         & $H_0$         & $H_1$   & $H_2$  & $H_3$  & $H_4$   & $H_5$   & $H_6$   & $H_7$           \\ \hline
$D_3$    & $\mathbb{Z}$   & $\mathbb{Z}^2$   & $\mathbb{Z}^2$    &        &         &         &       &                 \\ 
$D_4$    & $\mathbb{Z}$   & $\mathbb{Z}^2$   & $ \mathbb{Z}^4 \oplus \mathbb{Z}_2$   & $\mathbb{Z}^{5}$   &         &    &      &    \\ 
$D_5$    & $\mathbb{Z}$   & $0$  & $ \mathbb{Z}^2 \oplus \mathbb{Z}_2$   & $\mathbb{Z}^4$   & $\mathbb{Z}^4$   &    &     &       \\ 
$D_6$    & $\mathbb{Z}$   & $0$  & $\mathbb{Z}^2 \oplus  \mathbb{Z}_2^2$  & $ \mathbb{Z}^4 \oplus\mathbb{Z}_2$  & $\mathbb{Z}^4\oplus  \mathbb{Z}_2$   & $\mathbb{Z}^7$  &         &                 \\ 
$D_7$    & $\mathbb{Z}$   & $0$  & $ \mathbb{Z}^2\oplus \mathbb{Z}_2^2$  & $0$ & $ \mathbb{Z}^4\oplus \mathbb{Z}_2^3$  & $\mathbb{Z}^8$  & $\mathbb{Z}^6$   &                 \\ 
$D_8$    & $\mathbb{Z}$  &$0$  & $\mathbb{Z}_2^2$  & $0$  & $ \mathbb{Z}^4 \oplus \mathbb{Z}_2^2 \oplus \mathbb{Z}_3$ & $ \mathbb{Z}^8\oplus \mathbb{Z}_3$  & $ \mathbb{Z}^{10}\oplus \mathbb{Z}_2$  & $\mathbb{Z}^{13}$            \\ \hline
\end{tabular}
\vspace*{2mm}
\caption{Integral homology groups of $F_P$ for type  $D_n, n\leq 8$ }
\label{tab: disD}
\end{table}

We also tabulate  $H_i(F_{Q_0}; \mathbb{Z})$  for types $A_n$, $B_n$ and $D_n$ in  Table \ref{tab: redfibA}, Table \ref{tab: redfibB} and Table \ref{tab: redfibD}, respectively. The last column shows the Euler characteristics. The homology groups with question marks in these tables are unknown due to the memory limit of our computer.

A new result is that these homology groups in  Table \ref{tab: redfibA}, Table \ref{tab: redfibB} and Table \ref{tab: redfibD} are all torsion free. There are very few results on the integral homology group of $F_{Q_0}$ in the literature. Our results agree with Settepanella's stability-like theorem for $H_i(F_{Q_0}; \mathbb{Q})$ \cite{Set04}, and  in type $A_n$ agree with Settepanella's computations for $H_i(F_{Q_0}; \mathbb{Q})$ \cite[Table 2]{Set09}.

\begin{table}[h]
\begin{tabular}{c|cccccc|c}
\hline
      & $H_0(Q_0)$ & $H_1(Q_0)$ & $H_2(Q_0)$ & $H_3(Q_0)$ & $H_4(Q_0)$ & $H_5(Q_0)$ & $\chi(F_{Q_0})$ \\  \hline
 $A_1$  & $\mathbb{Z}$     &    &     &       &       & &         $1$ \\  
$A_2$ & $\mathbb{Z}$     & $\mathbb{Z}^4$    &     &       &       & &         $-3$  \\     
$A_3$ & $\mathbb{Z}$     & $\mathbb{Z}^7$    & $\mathbb{Z}^{18}$   &       &       & & $12$   \\ 
$A_4$ & $\mathbb{Z}$     & $\mathbb{Z}^9$     & $\mathbb{Z}^{28}$     & $\mathbb{Z}^{80}$  &       &  & $-60$  \\ 
$A_5$ & $\mathbb{Z}$     & $\mathbb{Z}^{14}$    & $\mathbb{Z}^{73}$    & $\mathbb{Z}^{206}$    & $\mathbb{Z}^{506}$  &  & $360$\\ 
$A_6$ & $\mathbb{Z}$     & $\mathbb{Z}^{20}$    & ?     & ?     & ?     & ? &$-2520$ \\ \hline
\end{tabular}
\vspace*{2mm}
\caption{Integral homology groups of $F_{Q_0}$  for type $A_n, n\leq 6$}
\label{tab: redfibA}
\end{table}

\begin{table}[h]
\centering
\begin{tabular}{c|ccccc|c}
\hline
      & $H_0(Q_0)$ & $H_1(Q_0)$ & $H_2(Q_0)$ & $H_3(Q_0)$ & $H_4(Q_0)$ & $\chi(F_{Q_0})$\\ \hline
$B_2$ & $\mathbb{Z}$     & $\mathbb{Z}^9$      &     &       &    & $-8$  \\       
$B_3$ & $\mathbb{Z}$      & $\mathbb{Z}^8$     & $\mathbb{Z}^{79}$    &       &    & $72$   \\ 
$B_4$ & $\mathbb{Z}$     & $\mathbb{Z}^{15}$     & $\mathbb{Z}^{73}$      & $\mathbb{Z}^{827}$     &     & $-768$  \\ 
$B_5$ & $\mathbb{Z}$     & $\mathbb{Z}^{24}$     &  ?  & ?   & ? & $9600$ \\   \hline
\end{tabular}
\vspace*{2mm}
\caption{Integral homology groups of $F_{Q_0}$ for type $B_n, n\leq 5$}
\label{tab: redfibB}
\end{table}

\begin{table}[h]
\centering
\begin{tabular}{c|ccccc|c}
\hline
      & $H_0(Q_0)$ & $H_1(Q_0)$ & $H_2(Q_0)$ & $H_3(Q_0)$ & $H_4(Q_0)$ & $\chi(F_{Q_0})$ \\ \hline
$D_2$ & $\mathbb{Z}$     & $\mathbb{Z}$     &     &       &  &$0$    \\ 
$D_3$ & $\mathbb{Z}$     & $\mathbb{Z}^7$     & $\mathbb{Z}^{18}$     &       &  &$12$    \\ 
$D_4$ & $\mathbb{Z}$     &  $\mathbb{Z}^{13}$     & $\mathbb{Z}^{69}$    & $\mathbb{Z}^{249}$  &   & $-192$     \\ 
$D_5$ & $\mathbb{Z}$    & $\mathbb{Z}^{18}$      &  $\mathbb{Z}^{133}$   & $\mathbb{Z}^{465}$   & $\mathbb{Z}^{3230}$   & $2880$ \\ \hline
\end{tabular}
\vspace*{2mm}
\caption{Integral homology groups of $F_{Q_0}$ for type $D_n, n\leq 5$}
\label{tab: redfibD}
\end{table}

\end{document}